\documentclass[12pt]{amsart}
\usepackage[top=1.5in, bottom=1.5in, left=1.4in, right=1.4in]{geometry}  
\usepackage{geometry}                % See geometry.pdf to learn the layout options. There are lots.
\usepackage{graphicx}
\usepackage{amssymb}
\usepackage{epstopdf}
\usepackage{color}
\usepackage{hyperref}
\usepackage{url}
\usepackage{verbatim}
\usepackage[lined,algonl,boxed,norelsize]{algorithm2e}
\title[Extension-lifting Bijections for Oriented Matroids]{Extension-lifting Bijections for Oriented Matroids}
\author{Spencer Backman, Francisco Santos, Chi Ho Yuen}
\date{\today}  % Activate to display a given date or no date

\numberwithin{equation}{section}
\theoremstyle{definition}

\newcounter{mainTheorem}
\newtheorem{theorem}{Theorem}[section]

\newtheorem{maintheorem}[mainTheorem]{Theorem}
\newtheorem{lemma}[theorem]{Lemma}
\newtheorem{definition}[theorem]{Definition}
\newtheorem{corollary}[theorem]{Corollary}

\newtheorem{remark}[theorem]{Remark}
\newtheorem{proposition}[theorem]{Proposition}

\newcommand{\Bcal}{{\mathcal B}}
\newcommand{\Ccal}{{\mathcal C}}

\newcommand{\Ecal}{{\mathcal E}}

\newcommand{\Ocal}{{\mathcal O}}
\newcommand{\Pcal}{{\mathcal P}}

\newcommand{\GLV}{\mathfrak{glv}}

\newcommand{\conv}{\operatorname{conv}}

\DeclareRobustCommand{\rchi}{{\mathpalette\irchi\relax}}
\newcommand{\irchi}[2]{\raisebox{\depth}{$#1\chi$}}

\usepackage{todonotes}

\newcommand{\paco}[1]{\todo[color=orange!30]{\footnotesize \rm #1 \\ \hfill --- P.}}
\newcommand{\Paco}[1]{\todo[size=\tiny,inline,color=orange!30]{#1 \\ \hfill --- P.}}

\begin{document}
\begin{abstract}
Extending the notion of geometric bijections for regular matroids, introduced by the first and third author with Matthew Baker, we describe a family of bijections between bases of an oriented matroid and special orientations.
These bijections are specified by a pair of circuit and cocircuit signatures coming respectively from a generic single-element lifting and extension.
We then characterize generic single-element liftings and extensions using these bijections.

%a family of bijections between bases of a regular matroid and the Jacobian group of the matroid was given. The core of the work is a geometric construction using zonotopal tilings that produces bijections between the bases of a realizable oriented matroid and the set of {\em $(\sigma,\sigma^*)$-compatible orientations} with respect to some {\em acyclic} circuit (respectively, cocircuit) signature $\sigma$ (respectively, $\sigma^*$). In this work, we extend this construction to general oriented matroids and circuit (respectively, cocircuit) signatures coming from generic single-element liftings (respectively, extensions). 

We also explain the relation of our work with the works of Gioan--Las Vergnas and Ding. Some implications in oriented matroid programming and oriented matroid triangulations are also discussed.
\end{abstract}

\maketitle

\tableofcontents

%\Chiho{Some convention to check: circuit vs signed circuit, $M$ vs $\underline{M}$, $M\setminus e$ vs $E\setminus\{e\}$, various notions of signatures}

\section{Introduction}

The number of spanning trees of a graph is an important numerical invariant that often enumerates other objects associated to the graph.
There is a long tradition of providing bijective proofs of such equinumerosity, e.g., proving Cayley's formula by Pr\"{u}fer sequences.
A particularly active direction in recent years is to construct bijections between spanning trees and special orientations or classes of orientations of the graph, a topic that has connections with the {\em chip-firing model} on graphs and their generalizations \cite{Bac_RR,Bac_PO}: there is a finite abelian group associated to the graph, known as {\em Jacobian} (also {\em critical group} or {\em sandpile group} in the literature), which can be defined using chip-firing and whose size is the number of spanning trees; the group has a canonical simply transitive action on these special orientations/orientation classes, so the bijections allow us to construct a simply transitive action of the Jacobian on the collection of spanning trees via composition.

A family of bijections, called 
``geometric bijections'' for their relation with polyhedral geometry (zonotopal tilings), was proposed by the first and third author together with Matthew Baker in \cite{BBY}. These bijections are particular well-behaved with respect to the said group action (see, for example, \cite[Section~5.6]{YCH_thesis} and \cite{DMTY}), and they naturally extend to {\em regular matroids}, where spanning trees are replaced by bases.
Since then, there have been several follow-up works along this direction, such as the work of McDonough which extended the tiling construction to (representable) {\em oriented arithmetic matroids} \cite{AM_GB}, and the works of Ding, including \cite{Ding_GB} which extended the bijections to other classes of subgraphs and gave a purely combinatorial proof of  bijectivity, and \cite{Ding_GB2} which we elaborate more in Section~\ref{sec:connection} and~\ref{sec:lawrence}.

In this paper we extract the oriented matroid theoretic essence of these geometric bijections and extend them to all oriented matroids. This is our Theorem~\ref{thm:main}, which we give a preview here; see Section~\ref{sec:prelim} for  details and precise definitions.

\subsection{Main results and idea of proof} 

Let $M$ be an oriented matroid with ground set $E$. 
A (generic) \emph{circuit signature} $\sigma$ on $M$ 
is a way to pick, for each circuit $\underline{C}$ of the underlying matroid of $M$, one of the two signed circuits of $M$ supported on $\underline{C}$.
% \spencer{I think it would helpful to explain how this language is unorthodox.}% 
%\chiho{Removed the footnote.}
%\footnote{The precise definition of a circuit signature is an antipodal map from signed circuits to $\{+,-,0\}$.
%The unorthodox usage here is equivalent to a signature (only) when it is {\em generic}.
%}
%{The usage here is unorthodox and only works when the lifting is generic, the precise definition will be given in Section~\ref{sec:prelim}. }
As an example, any (generic) {\em single-element lifting} of $M$, i.e., any oriented matroid $\widetilde{M}$ such that $M=\widetilde{M}/g$ for some element $g$ only lying in spanning circuits of $\widetilde{M}$, 
induces a generic circuit signature.
 Dually, a generic {\em single-element extension} of $M$, i.e., an oriented matroid $M'$ such that $M=M'\setminus f$ for some element $f$ only lying in spanning cocircuits of $M'$, induces a generic {\em cocircuit signature} $\sigma^*$ that picks out a signed cocircuit supported on each cocircuit.

An {\em orientation} of $M$ is a map $\Ocal:E\rightarrow\{+,-\}$.
%$\Ocal$ is {\em conformal} with a signed circuit or cocircuit $C$ if $\Ocal(e)=C(e)$ for every $e\in\underline{C}$. 
One can equivalently speak  about a {\em reorientation} $_{-A}M$ of $M$ along a subset $A\subset E$ of elements. The two points of view are equivalent by letting $\Ocal(e) = -$ if and only if $e\in A$.
%; in this description $\Ocal$ is conformal with $C$ if $C^- = \underline{C} \cap A$. A circuit or cocircuit $C\subset E$ is {\em compatible} with $\Ocal$ if one of the signed versions of $C$ is conformal with $\Ocal$.
We say that an orientation $\Ocal$ is compatible with a generic circuit signature $\sigma$ if every signed circuit $C$ conformal with $\Ocal$ is the signed circuit picked out by $\sigma$ for $\underline{C}$; we similarly define the compatibility with a generic cocircuit signature $\sigma^*$.
When both things happen we say that $\Ocal$ is {\em $(\sigma,\sigma^*)$-compatible}.

Recall that for a basis $B$ of $M$ and an arbitrary element $e$, there is either a \emph{fundamental circuit} $C(B,e)$ (if $e\not\in B$) or \emph{fundamental cocircuit} $C^*(B,e)$ (if $e\in B$): the unique circuit contained in $B\cup e$, or the unique cocircuit disjoint with $B\setminus e$. This motivates the following way of inducing orientations of $M$:

\begin{definition}
\label{defi:beta}
    Let $M$ be an oriented matroid. Let $\sigma$ and $\sigma^*$ be a generic circuit and cocircuit signature of $M$, respectively. 

     Given a basis $B$, let $\Ocal(B)$ be the orientation of $M$ in which we orient each $e \not\in B$ according to its orientation in $\sigma(C(B,e))$ and each $e \in B$ according to its orientation in $\sigma^*(C^*(B,e))$. We denote
    $\beta_{\sigma,\sigma^*}$ the map 
    \[
    \begin{array}{ccc}
    \beta_{\sigma,\sigma^*} : \{\text{bases of }M\} &\longrightarrow &\{\text{orientations of }M\}\\
    B &\longmapsto &\Ocal(B).
    \end{array}
    \]
\end{definition}

Our main result in this paper is:

\begin{maintheorem} \label{thm:main}
Let $M$ be an oriented matroid, and let $\sigma$ and $\sigma^*$ be the generic circuit and cocircuit signatures induced from a generic single-element lifting and extension, respectively.
%Given a basis $B$, let $\Ocal(B)$ be the orientation of $M$ in which we orient each $e \not\in B$ according to its orientation in $\sigma(C(B,e))$ and each $e \in B$ according to its orientation in $\sigma^*(C^*(B,e))$. 
%
Then the map 
$
\beta_{\sigma,\sigma^*}:B \mapsto \Ocal(B)
$
is a bijection between the set of bases of $M$ and the set of $(\sigma,\sigma^*)$-compatible orientations of $M$.
\end{maintheorem}

It turns out that the bijections in Theorem~\ref{thm:main} are so closely related to liftings and extensions of oriented matroids that they characterize generic signatures coming from liftings or extensions. More precisely:

\begin{maintheorem} 
\label{thm:characterization}
Let $\sigma$ be a generic circuit signature of an oriented matroid $M$.
The following are equivalent:
\begin{enumerate}
\item $\sigma$ is the circuit signature induced by some single-element lifting of $M$.
\item $\beta_{\sigma,\sigma^*}$ is a bijection between bases of ${M}$ and $(\sigma,\sigma^*)$-compatible orientations, for every generic cocircuit signature $\sigma^*$ of $M$ induced by generic single-element extension.
\item $\beta_{\sigma,\sigma^*}$ maps bases of $M$ to $\sigma$-compatible orientations for every generic lexicographic cocircuit signature $\sigma^*$ of $M$.
%\chiho{To Paco: The proof below still works if we only say ``orientations compatible with $\sigma$'', right? A: YES. CHANGED STATEMENT}
\end{enumerate}
\end{maintheorem}

Let us remark that part (3), saying that  ``lexicographic signatures suffice'', has a similar flavor as \cite[Theorem~3.5]{Santos_book}.

\begin{remark} \label{rem:acyclic}
    As mentioned above, our work was originally motivated by \cite{BBY}, which proves Theorem~\ref{thm:main} for realizable oriented matroids. Let us explain this connection.
    
    Let $A$ be an $r\times n$ real matrix of full row-rank realizing $M$. An extension (respectively, lifting) of $A$ is an $(r+1)\times n$ matrix $\widetilde A$ (respectively, an $r\times (n+1)$ matrix $A'$) restricting to $A$ by deletion of its last row (respectively, column). Each such extension/lifting defines a circuit/cocircuit signature of $M$, and the signatures that can be obtained this way are called \emph{acyclic} in \cite{BBY}.
    Theorem 1.4.1 in \cite{BBY} is exactly Theorem~\ref{thm:main} for the case of acyclic signatures of a realized oriented matroid. The proof there uses (fine) zonotopal tilings
    of the zonotope of $A$ which, by the Bohne--Dress Theorem, are equivalent to (generic) liftings of $M$~\cite{Bohne_thesis, RZ_BD_thm}. 
\end{remark}

\medskip

Our proof of Theorem~\ref{thm:main} uses the theory of {\em oriented matroid programs} (OMPs) developed by Bland and Lawrence~\cite{Bland-OMP,FL-OM}; see details in the next section, and in \cite[Chapter 10]{BLSWZ_book}.
The intuition is as follows. 

Let $\widetilde M$ and $M'$ be a generic lifting and extension of $M$ inducing the signatures $\sigma^*$ and $\sigma$.
By the Topological Representation Theorem of Folkman and Lawrence \cite{FL-OM},
we can picture $\widetilde{M}$ (together with its distinguished element $g$) as an {\em affine pseudohyperplane arrangement} 
with $g$ as the hyperplane at infinity. Cells of the arrangement correspond to covectors of $\widetilde M$, those containing $g$ lying in the affine chart and those not containing $g$ lying at infinity.

In this picture, each region $R$ of the arrangement (or tope of $\widetilde M$) corresponds to a $\sigma$-compatible orientation of $M$, namely the orientation that makes $R$ the positive region of the arrangement.
See the left part of Figure~\ref{fig:PSA}, where an affine arrangement (and its corresponding zonotopal tiling) are shown.

%In the realizable case, such arrangement can be thought as the dual of the zonotopal tiling used in \cite{BBY}; this phenomenon is related to the Bohne--Dress theorem on single-element liftings of realizable oriented matroids \cite{Bohne_thesis, RZ_BD_thm}.

\begin{figure}
\centering
\includegraphics[width=0.45\textwidth]{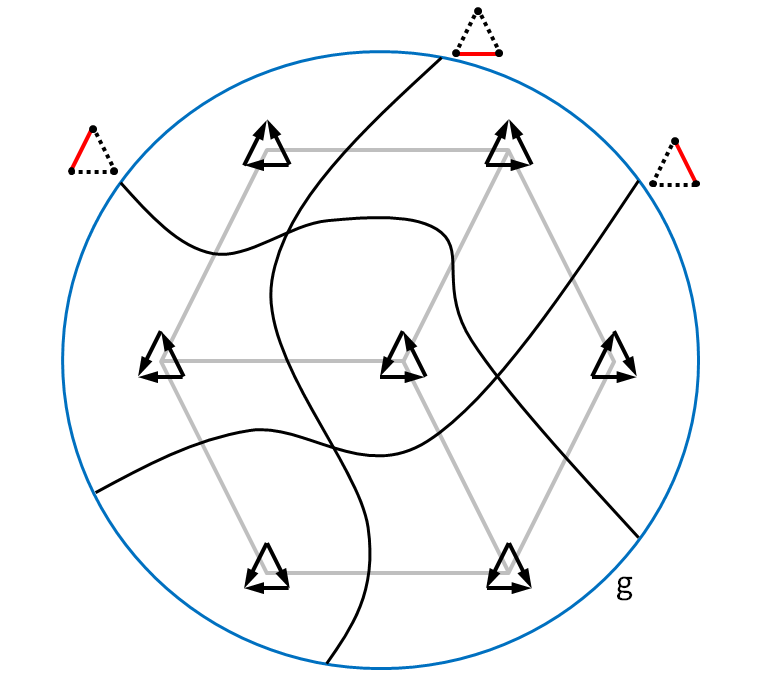}
\includegraphics[width=0.45\textwidth]{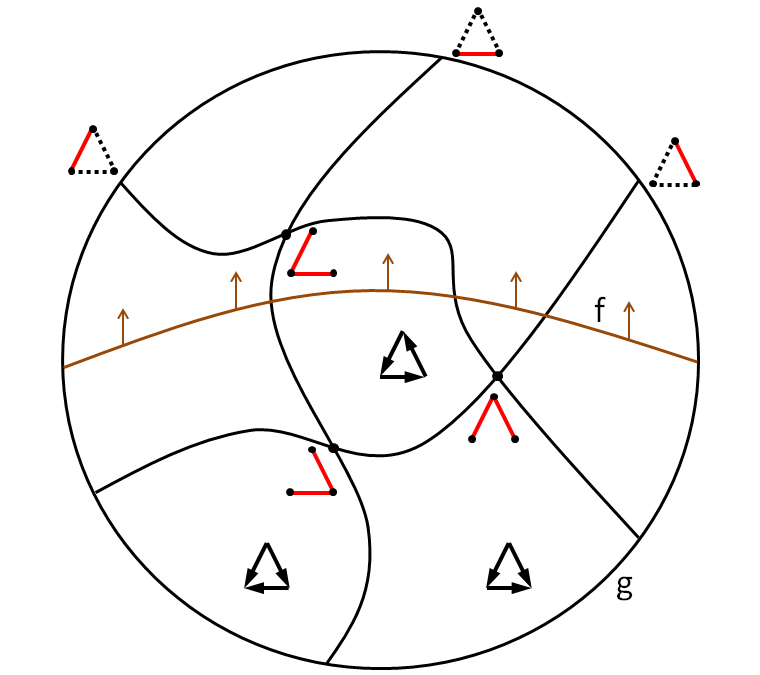}
\caption{Left: The affine pseudohyperplane arrangement of $M(K_3)$ (the three curves represent the elements of $M$), together with the extra elements $g$. The regions are labeled by $\sigma$-compatible orientations of $M$.
Right: The new curve represents $f$. There are three regions whose optima with respect to $f$ are bounded, and each of these optima is the intersection of curves that form a basis of $M$.}
\label{fig:PSA}
\end{figure}

The distinguished element $f$ of the extension $M'$ can be included into the picture as an ``increasing direction'' or ``objective function'', with respect to which we consider the optimum of each region; see the right part of Figure~\ref{fig:PSA}. 
In Section \ref{sec:main} (Theorem~\ref{thm:main_actual}), we prove that 
\begin{quote}
    the regions of the arrangement which are {\em bounded} with respect to $f$ (that is, those 
    whose optima are not lying on $g$) are precisely the $(\sigma,\sigma^*)$-compatible orientations
\end{quote}
   and (by genericity of $M'$)
\begin{quote}
    bounded vertices of the arrangement (that is, cocircuits of $\widetilde M$ containing $g$) correspond bijectively to bounded regions, by assigning to each bounded region its optimum vertex. 
\end{quote}  
    Finally, by genericity of $\widetilde{M}$, 
\begin{quote}
    (bounded) vertices of the arrangement correspond bijectively to bases of $M$. 
\end{quote}  
     These three bijections together give the bijection stated in Theorem~\ref{thm:main}. It follows by construction that the map  from bases to orientations so obtained coincides with the map $\beta_{\sigma,\sigma^*}$. 
     
     Let us mention however that, although the pseudo-hyperplane arrangement image is useful for intuition, our proof does not use or need the Topological Representation Theorem. Instead, we directly construct the bijection between $(\sigma,\sigma^*)$-compatible orientations and bounded vertices of the arrangement (that is, cocircuits of $\widetilde M$ not vanishing at $g$) using the \emph{Fundamental Theorem of Oriented Matroid Programming} (Theorem~\ref{thm:OMP}).

\subsection{Connections with Other Works} \label{sec:connection}

As mentioned at the beginning, there have been several works on constructing bijections between bases and orientations in recent years. Two of them, by Gioan--Las Vergnas and by Ding, respectively, are of particular relevance here. We explain how these works are related to our construction and highlight some technical/conceptual differences. We consider clarifying this picture to be a secondary contribution of this paper.

\medskip

There is a series of works by Gioan and Las Vergnas on {\em active bijections}~\cite{gioan2015survey, gioan2022tutte, gioanactive, gioan2004active, GLV_Uniform, gioan2005activity, gioan2007fully, GLV_AB1, gioan2009linear, GLV_AB2a, GLV_AB2b, gioan2019active, gioan2018computing, las1984correspondence}, which began with the PhD thesis  of Gioan \cite{Gioan_thesis}. Given an oriented matroid $M$ and an ordering of its elements, one can define the {\em internal and external activities} of a basis and of an orientation. The \emph{active bijection} maps bases (together with extra specifications) to orientations in an activity preserving manner (hence their name). See Theorem~\ref{thm:G_LV_bij} for a more precise statement.

A major component in the construction by Gioan and Las Vergnas is to describe  the bijection for $(1,0)$-bases (bases with a single internally active element and no externally active ones). For the proof of this case, the machinery of oriented matroid programming is also applied, in a way quite similar to ours.

In Section~\ref{sec:activity} we show that the bijection of Theorem~\ref{thm:main} for the original oriented matroid $M$ with a given generic extension $M'$ and lift $\widetilde M$ coincides with the active bijection of Gioan and Las Vergnas for an oriented matroid $\widetilde M'$ that simultaneously extends/lifts $M'$ and $\widetilde M$ and in which $g$ and $f$ are the first and second element in the order. The existence of $\widetilde M'$ is a result of \cite{Santos_book} that we also use for the proof of Theorem~\ref{thm:main} (see our Lemma~\ref{lem:common_ex_lifting}).
This also establishes a connection between our work with the classical theorem of Greene and Zaslavsky that the number of bounded regions in an affine pseudohyperplane arrangement equals the {\em beta invariant} of the corresponding matroid (independent of the choice of pseudohyperplane at infinity) \cite[Theorem D]{GZ_HPZono}: the bounded regions of this auxiliary oriented matroid (with $g$ being at infinity) correspond to the regions of the oriented matroid program considered in the proof of Theorem~\ref{thm:main}, whose optima with respect to $f$ are bounded.
%Our Theorem~\ref{thm:main} can be thought as counting regions with respect to another type of boundedness, and again the count is independent of the choice of $f$ and $g$ (as long as the choice is generic). 

%\Chiho{My suggestion to revise the above paragraph}
%\Paco{edited it}
However, regarding our bijection as a special case of the active bijection would be an oversimplification, since the two bijections have different input data and description, and since 
many extensions  of Theorem~\ref{thm:main} cannot be easily formulated using the language of active bijections; this includes  Theorem~\ref{thm:characterization}, the construction of Jacobian actions for the regular matroid case, and the relation with oriented matroid triangulations (see below).
Conversely, our setting cannot capture many features of active bijections either.
It is thus appropriate to say that our work gives an alternative way of applying the Main Theorem of OMP to bijective questions.

\medskip

In \cite{Ding_GB2}, Ding unified geometric bijections (as in the original \cite{BBY}) and the family of bijections by K\'{a}lm\'{a}n--T\'{o}thm\'{e}r\'{e}sz \cite{TL_GB} for regular matroids\footnote{The work of K\'{a}lm\'{a}n--T\'{o}thm\'{e}r\'{e}sz  concerns {\em hypergraphs} and their associated {\em root polytopes}. Ding transferred the construction to Lawrence polytopes of regular matroids.}, using the insight that the data needed to define each family of bijections can be interpreted in terms of dissections of the {\em Lawrence polytope} $\Lambda(M)$ of a regular matroid $M$ (and/or its dual).
In particular, he suggested the notion of {\em triangulating signatures} that capture the data of arbitrary triangulations of $\Lambda(M)$, generalizing acyclic signatures (see Remark~\ref{rem:acyclic}) which induce regular triangulations.
He then proved in \cite[Theorem~1.20]{Ding_GB2} that a pair of triangulation signatures yields a bijection in the same way as \cite[Theorem~1.4.1]{BBY}.

We show that Ding's bijections {\em in the triangulation setting} are special cases of Theorem~\ref{thm:main} by showing that his \emph{triangulating signatures} are precisely signatures induced by generic single-element liftings/extensions, and they induce triangulations of Lawrence polytopes using the abstract formalism of oriented matroid triangulations developed by the second author \cite{Santos_book}.
%As a by-product, we obtain a proof of the generic case of the Bohne--Dress Theorem, namely a bijection between fine zonotopal tilings of a zonotope and generic liftings of its underlying oriented matroid.
%\chiho{We currently only sketched the proof, but actually use the result of Paco that is morally Bohne--Dress itself.}

Even if the relation between single-element lifts and triangulations of Lawrence polytopes holds for arbitrary oriented matroids, other
ingredients in Ding's proof (or the very formulation) of the results of \cite{Ding_GB2} depend critically on the properties of regular matroids (e.g., realizability), and do not extend, as far as we can see, to an alternative proof of our results for general oriented matroids.
Nevertheless, the bijections in \cite{Ding_GB2} are derived from not just a pair of triangulations of the Lawrence polytope, but also when one of them is a dissection, a notion strictly more general than triangulations and with no abstraction in the context of oriented matroids at this moment; moreover, the formalism (using {\em atlases}) and arguments in his work are novel and of independent interest. Hence our work does not subsume his either.

\medskip

We end this section mentioning another result conceptually, but not technically, similar to ours.
Given a {\em strong map} between oriented matroids $M_1\rightarrow M_2$ on the same ground set, Las Vergnas gave a formula to count the number of orientations that are acyclic in $M_1$ and totally cyclic in $M_2$ \cite{LV_ATCO}. Theorem~\ref{thm:main} has a similar flavour in view of Lemma~\ref{lem:general_CCMO}, although we note that the map $\widetilde{M}\rightarrow M'$ is not a strong map in general; 
indeed, while an extension followed by a contraction of the new elements gives rise to a strong map, $\widetilde{M}\rightarrow M'$ can be thought of (by Lemma~\ref{lem:common_ex_lifting}) as a single-element extension followed by contracting a {\em different} element (or equivalently, the map is a contraction followed by an extension).

\subsection{Background and Organization of the Paper}

Theorem~\ref{thm:main} has been the content of an earlier version of this manuscript, accepted as an extended abstract of FPSAC 2019 under the title {\em Topological Bijections for Oriented Matroids}.
The present paper adds substantial new material to it, including Theorem~\ref{thm:characterization} and an update on the advances along this direction in recent years. We also decided to change the name of our bijections to ``extension-lifting bijections'' since, as mentioned above, we use the topological representation of oriented matroids as intuition but our results and proofs do not rely on it.

%The structure of the paper is as follows. 
In Section~\ref{sec:prelim} we recall some basic facts from oriented matroid theory, mostly related to extensions, liftings, and oriented matroid programs. Section~\ref{sec:main} contains the proof of Theorems~\ref{thm:main} and~\ref{thm:characterization}, and Sections~\ref{sec:activity} and \ref{sec:lawrence} show how our results connect to the active bijections by Gioan and Las Vergnas, and to the work of Ding on triangulations of Lawrence polytopes.

\section{Preliminaries}
\label{sec:prelim}

We assume the reader is familiar with the basic definitions in oriented matroid theory, and we refer to \cite{BLSWZ_book} for details and notation. 
We here recall some notions and results that are crucial in our proofs.

Let $M$ be an oriented matroid of rank $r$ on ground set $E$. The set of bases of $M$ will be denoted by $\Bcal(M)$, and the set of signed circuits (respectively, signed cocircuits) of $M$ will be denoted by $\Ccal(M)$ (respectively, $\Ccal^*(M)$). The support of a signed subset $X$ will be denoted by $\underline{X}$. Whenever we speak of a circuit without the adjective ``signed'', we always mean a circuit of the underlying ordinary matroid; same for a cocircuit of $M$. We often write $M\cup f$, $M\setminus f$, etc instead of $M\cup\{f\}$, $M\setminus\{f\}$, etc for simplicity.
%, and the underlying matroid of an oriented matroid $M$ will be denoted by $\underline{M}$

\begin{definition}
An oriented matroid $M'$ of rank $r$ is a {\em single-element extension} of $M$ if the ground set of $M'$ is $E\sqcup f$ for some new element $f$ and $M=M'\setminus f$. Dually, $\widetilde{M}$ of rank $r+1$ is a {\em single-element lifting} of $M$ if the ground set of $\widetilde{M}$ is $E\sqcup g$ for some new element $g$ and $M=\widetilde{M}/g$.

%\paco{added definition of generic}
The extension (respectively, lifting) is \emph{generic} if every circuit of $M'$ (respectively cocircuit of $\widetilde M$) containing $f$ (respectively, $g$) is spanning (respectively, has independent complement). 
%\paco{added the "equivalently"}
%Equivalently, if every cocircuit of $M'$ (respectively circuit of $\widetilde M$) not containing $f$ (respectively, $g$) has independent complement (respectively, is spanning).

\end{definition}

A {\em cocircuit signature} is a map ${\sigma}^*:\Ccal^*(M)\rightarrow\{+,-,0\}$ satisfying $\sigma^*(-D)=-\sigma^*(D)$; it is {\em generic} if the image lies in $\{+,-\}$.
Let $M'$ be a single-element extension of $M$. For every signed cocircuit $D$ of $M$ there exists a unique signed cocircuit $D'$ of $M'$ such that $D'|_E=D$. Therefore we can define a cocircuit signature associated to the extension by setting ${\sigma}^*(D):=D'(f)$. Clearly, the extension is generic if and only if $\sigma^*$ is generic. In such case we can understand ${\sigma}^*$ 
%induces a {\em cocircuit signature} $\sigma^*$ of $M$ that 
as selecting one of the two signed cocircuits of $M$ supported on each cocircuit of $M$, namely the one that extends to have $f$ on its positive side.
Thus, we often abuse notation and write $\sigma^*(\underline{D})$ to be such a signed cocircuit for a cocircuit $\underline{D}$ of $M$ and say that such a signed cocircuit is {\em compatible} with $\sigma^*$; this is the terminology being used in some literature such as \cite{BBY, Ding_GB, Ding_GB2}, and should cause no ambiguity as the input is now a (unsigned) cocircuit.

Dually, every generic single-element lifting induces a generic {\em circuit signature} $\sigma$ that sends each circuit $\underline{C}$ of $\underline{M}$ to the signed circuit $\sigma(\underline{C})$ of $M$ with that support that extends to have $g$ on its positive side.

%\paco{(Re)moved acyclic signatures}
A particularly important class of single-element extensions/liftings are the lexicographic extensions/liftings.

\begin{definition} \label{def:lexi}
A set of {\em lexicographic data} is an ordered subset $E':=(e_1,\ldots,e_k)$ of $E$ together with a sequence of signs $s_1,\ldots,s_k\in\{+,-\}$. If $E=E'$ we say the data is \emph{full}.

The {\em lexicographic cocircuit signature} induced by such data is defined as follows: for each signed cocircuit $D$,  ${\sigma}^*(D)=0$ if $\underline{D}\cap E'=\emptyset$, or otherwise ${\sigma}^*(D)= s_i\cdot D(e_i)$ where $i$ is the first index with $e_i \in\underline{D}\cap E'$.
By definition, this signature is generic if and only if $E'$ is spanning. In this case $\sigma^*$ sends each cocircuit $\underline{D}$ to the signed cocircuit $D$ such that $D(e_i)=s_i$, where $e_i$ is again the first element in $\underline{D}\cap E'$.

Define the {\em lexicographic circuit signature}  analogously. It is generic if and only if $E\setminus E'$ is independent. 
\end{definition}

%Observe that for generic lexicographic extensions and liftings there is no loss of generality in assuming $E=E'$, except the last few elements in it do not affect the result.

\begin{lemma}[Las Vergnas, see~\protect{\cite[Definition~7.2.3]{BLSWZ_book}}]
Lexicographic cocircuit signatures (respectively, circuit signatures) are extension signatures (respectively, lifting signatures).
\end{lemma}

%\begin{example}
%In \cite{Bernardi}, Bernardi introduced a combinatorial process for graphs embedded on orientable surfaces, and use it to define a family of bijections between spanning trees of a graph and special {\em chip configurations}.
%Briefly speaking, given a spanning tree of an embedded graph, one ``walks'' around the tree according to the embedding, and puts chips at the vertices as per the tour (or in the language of this note, orients the edges as per the tour).
%By \cite[Example~2.4.3.]{BBY}, when the embedded graph is a plane graph, Bernardi's bijections are geometric bijections, which from the above proposition are topological bijections as well.
%Hence, in some sense, the previously proven results on duality of planar Bernardi bijections \cite{BW_torsor} (and {\em a posteriori} the parallel results on planar {\em rotor-routing} \cite{CGMPWY}) are oriented matroid duality in disguise.
%\end{example}

As said in the introduction, an {\em orientation} of an oriented matroid $M$ is a map $\Ocal:E\rightarrow\{+,-\}$ and it can be interpreted  as a {\em reorientation} $_{-A}M$ of $M$ along a subset $A\subset E$ of elements.
We say that $\Ocal$ is {\em conformal} with a signed circuit or signed cocircuit $C$ if $\Ocal(e)=C(e)$ for every $e\in\underline{C}$. (This agrees with the usual notion of conformality for arbitrary sign vectors); we also say $\Ocal$ is conformal with a circuit or cocircuit $\underline{C}$ if $\Ocal|_{\underline{C}}$ is a signed circuit or signed cocircuit of $M$.
%\chiho{``Signed'' is more precise here, otherwise it should be ``$\underline{C}$ is a positive circuit of $_{-A}M$''.}
%if there is no element $e$ where $\Ocal$ and $C$ have (non-zero) opposite signs. Since $\Ocal$ is never zero, this is equivalent to 
% A circuit or cocircuit $C\subset E$ is {\em compatible} with $\Ocal$ if one of the signed versions of $C$ is compatible with $\Ocal$.

\begin{definition} \label{def:CCMO}
Let $\sigma$ (respectively, $\sigma^*$) be the circuit (respectively, cocircuit) signature induced by some generic single-element lifting $\widetilde{M}$ (respectively, extension $M'$). Then an orientation $\Ocal$ of $M$ is {\em $(\sigma,\sigma^*)$-compatible} if $\sigma$ (respectively, $\sigma^*)$ is positive at every signed circuit (respectively, cocircuit) conformal with $\Ocal$.
%is oriented according to $\sigma$ (respectively, $\sigma^*$). 
The set of $(\sigma,\sigma^*)$-compatible orientations of $M$ is denoted by $\rchi(M;\sigma,\sigma^*)$.
\end{definition}

For the proof of Theorem~\ref{thm:main} we use the machinery of oriented matroid programming. We now recall the Main Theorem of Oriented Matroid Programming, giving some essential definitions along the way.

\begin{theorem} \cite[Theorem 10.1.13]{BLSWZ_book} \label{thm:OMP}
Let $M$ be an oriented matroid with two distinguished elements $f\neq g$, such that $f$ is not a coloop and $g$ is not a loop; the data $(M,g,f)$ specifies an {\em oriented matroid program}.

Suppose the program is (equivalent formulations by \cite[Theorem 10.1.9]{BLSWZ_book})
\begin{enumerate}
    \item {\em Feasible}: there exists a signed cocircuit $Y$ (a ``feasible region'') of $M$ such that $Y(e)\geq 0 \  \forall e\neq f$ and $Y(g)=+$, and
    \item {\em Bounded}: there exists a signed circuit $C$ (a ``bounded cone'') such that $C(e)\geq 0\  \forall e\neq g$ and $C(f)=+$.
\end{enumerate}

Then the program has a solution $Y$, which is a covector of $M$ that is:
\begin{enumerate}
    \item {\em Feasible}: $Y(e)\geq 0 \  \forall e\neq f$ and $Y(g)=+$, and
    \item {\em Optimal}: for every covector $Z$ such that $Z(f)=+,\ Z(g)=0$, we have $(Y\circ Z)(e)<0$ for some $e\neq f,g$.
\end{enumerate}
\end{theorem}

For Theorem~\ref{thm:characterization}
we need the following characterization of signatures of extensions/liftings. This was proved by Las Vergnas~\cite{LasVergnas} and is reproduced also in \cite[Theorem 7.1.8]{BLSWZ_book} and \cite[Lemma 1.3]{Santos_book}. (Only the case of extensions is stated in these references, but the case of liftings follows by duality.)

\begin{lemma}[Las Vergnas]
\label{lemma:3elements}
Let $\sigma: \Ccal(M)\to \{+,-,0\}$ and $\sigma^*: \Ccal^*(M)\to \{+,-,0\}$ be signatures defined on the set of signed circuits and signed cocircuits of an oriented matroid $M$, respectively. Then:
\begin{enumerate}
    \item $\sigma^*$ is the cocircuit signature of a single-element extension if and only if it is so when restricted to every uniform rank two minor on three elements.
    \item $\sigma$ is the circuit signature of a single-element lifting if and only if it is so when restricted to every uniform rank one minor on three elements.
\end{enumerate}
\end{lemma}

In this statement restricting $\sigma$ or $\sigma^*$ means the following. Let $U= M/B\backslash A$ be a minor of $M$, where $A$ and $B$ are disjoint subsets of the set $E$ of elements of $M$. Then, every signed cocircuit $D$ of $U$ lifts to a unique signed cocircuit $\widetilde D$ of $M$ whose support is contained in $E\setminus B$, and every signed circuit $C$ of $U$ extends to a unique signed circuit $ C'$ of $M$ whose support is contained in $E\setminus A$. Restricting $\sigma$ and $\sigma^*$ to $U$ means giving $C$ and $D$ the signs that $C'$ and $\widetilde D$ get in $M$.

\section{Proof of the Main Results} \label{sec:main}

Throughout this section, let $M$ be an oriented matroid on ground set $E$, and let $M'$ (respectively, $\widetilde{M}$)  be a generic single-element extension (respectively, lifting) of $M$ on ground set $E\cup f$ (respectively, $E\cup g$).

Theorem~\ref{thm:main} will be deduced from the following theorem, which constructs the inverse of the map $\beta_{\sigma,\sigma^*}$.
We use the following notation.
%\paco{moved this notation here}
For an orientation $\Ocal$ of $M$, $\Ocal'_-$ denotes the orientation of $M'$ such that $\Ocal'_-|_E=\Ocal$ and $\Ocal'_-(f)=-$; dually, $\widetilde{\Ocal}_-$ is the orientation of $\widetilde{M}$ such that $\widetilde{\Ocal}_-|_E=\Ocal$ and $\widetilde{\Ocal}_-(g)=-$.
Given a sign vector $X$ of $M$, when we consider $M'$ (respectively, $\widetilde{M}$), $(X\ \epsilon), \epsilon\in\{+,-,0\}$ is understood as a sign vector that agrees with $X$ over $E$ and is equal to $\epsilon$ over $f$ (respectively, $g$).

\begin{theorem} \label{thm:main_actual}
For every $\Ocal\in\rchi(M;\sigma,\sigma^*)$, there exists a unique basis $B\in\Bcal(M)$ such that $B\cup f$ is a circuit conformal with $\Ocal'_-$ and $(E\setminus B)\cup g$ is a cocircuit conformal with $\widetilde{\Ocal}_-$.
\end{theorem}

As explained in the introduction, such a basis corresponds to the optimum (with respect to $f$) of the region corresponding to $\Ocal$ in the pseudohyperplane arrangement representing $\widetilde M$.

We start with a few lemmas.

\begin{lemma} \label{lem:generic_ex_circuit}
The set of circuits of $M'$ containing $f$ is $\{B\cup f:B\in\Bcal(M)\}$. Dually, the set of cocircuits of $\widetilde{M}$ containing $g$ is $\{(E\setminus B)\cup g:B\in\Bcal(M)\}$.
\end{lemma}

\begin{proof}
Let $B\in\Bcal(M)$. $B$ is also a basis of $M'$, since $M$ and $M'$ have the same rank and independent sets of $M$ are independent in $M'$. 
Now,
%Let $B\in\Bcal(M)$. We first claim that $B$ is also a basis of $M'$. Since every circuit of $M'$ not containing $f$ is a circuit of $M$, $B$ is independent in $M'$; since every circuit of $M$ is a circuit of $M'$, $B\cup\{e\}$ is dependent in $M'$ for any $e\in E\setminus B$. So if $B$ is not a basis of $M'$, it must be the case that $X:=B\cup\{f\}$ is a basis of $M'$. In such case, $B=X\setminus\{f\}$ avoids the fundamental cocircuit $D'$ of $f$ with respect to $X$ in $M'$. Since $M'$ is generic, $f$ is not an isthmus and $D'\setminus\{f\}$ contains a cocircuit $D''$ of $M$, now $B$ avoids the cocircuit $D''$ in $M$, contradicting the basic property of bases.
the fundamental circuit $C(B,f)$ must be spanning, by definition of genericity for $f$, so it strictly contains a basis. This implies $C(B,f) = B\cup \{f\}$. In particular, for every basis $B$ of $M$, $B\cup \{f\}$ is a circuit of $M'$.
%is the whole of $B\cup f$. Suppose not, pick an arbitrary $e\in (B\cup f)\setminus C'$ and let $D$ be the fundamental cocircuit of $e$ with respect to $B$ in $M$. On one hand, $D':=D\cup f$ is a cocircuit of $M'$ as the extension is generic, so $D'$ must be the fundamental cocircuit of $e$ with respect to $B$ in $M'$. On the other hand, since $e\not\in C'=C(B,f)$, $f$ cannot be in $D'=C^*(B,e)$, a contradiction. This shows $\{B\cup f:B\in\Bcal(M)\}$ are all circuits of $M'$.

Conversely, if $C'$ is a circuit of $M'$ containing $f$ then it is spanning. Being a circuit, removing one element (e.g. $f$) from it rank is preserved, so $C'\setminus \{f\}$ is indeed spanning and independent, i.e., a basis.
%
%. Then $Y:=C'\setminus f$ is independent in $M'$ thus in $M$. If $Y$ is not a basis of $M$, then it is properly contained in some $B\in\Bcal(M)$, but by the above containment, $B\cup f$ is a circuit of $M'$ properly containing $C'$, a contradiction. 

The dual statement is  proven similarly.
\end{proof}

\begin{lemma} \label{lem:general_CCMO}
An orientation $\Ocal$ of $M$ is $\sigma^*$-compatible if and only if $\Ocal'_-$ is totally cyclic. Dually, $\Ocal$ is $\sigma$-compatible if and only if $\widetilde{\Ocal}_-$ is acyclic.
\end{lemma}

\begin{proof}
Suppose $\Ocal'_-$ is conformal with some signed cocircuit $D'$. By \cite[Proposition 7.1.4 (ii)]{BLSWZ_book}, $D:=D'|_E$ is either (i) a signed cocircuit of $M$, in which $f\in\underline{D'}$, or (ii) equal to the conformal composition $D_1\circ D_2$ of signed cocircuits of $M$, in which $\sigma^*(D_1)=-\sigma^*(D_2)\neq 0$. For case (i), $D$ is a signed cocircuit conformal with $\Ocal$, but it is not compatible with $\sigma^*$ as $D'(f)=\Ocal'_-(f)=-$; for case (ii), both $D_1,D_2$ are conformal with $\Ocal$, but exactly one of them is not compatible with $\sigma^*$ as $\sigma^*(D_1)=-\sigma^*(D_2)$.

Conversely, if $D$ is a signed cocircuit conformal with $\Ocal$ but not compatible with $\sigma^*$, then $(D\ -)$ is a signed cocircuit of $M'$ that is conformal with $\Ocal'_-$, hence $\Ocal'_-$ is not totally cyclic. The dual statement can be proven similarly.
\end{proof}

Using the above lemmas, we can give an alternative description of the map $\beta_{\sigma,\sigma^*}$, matching the statement of Theorem~\ref{thm:main_actual}.

\begin{proposition} \label{prop:beta_EL_pos}
Let $B$ be a basis of $M$ and let $\Ocal=\beta_{\sigma,\sigma^*}(B)$. Then $B\cup f$ is a circuit conformal with $\Ocal'_-$ and $(E\setminus B)\cup g$ is a cocircuit conformal with $\widetilde{\Ocal}_-$.
\end{proposition}

\begin{proof}
By Lemma~\ref{lem:generic_ex_circuit}, $X:=B\cup f$ is a circuit of $M'$. Denote by $C$ the signed circuit of $M'$ whose support is $X$ and satisfies $C(f)=-$. For every $e\in B$, let $D_e$ be the fundamental cocircuit of $e$ with respect to $B$ in $M$, oriented according to $\sigma^*$. By the definition of $\sigma^*$, the signed subset $D_e':=(D_e\ +)$ is a signed cocircuit of $M'$, and $X\cap \underline{D_e'}=\{e,f\}$. By the orthogonality of signed circuits and cocircuits as well as the fact that $D_e'(f)=-C(f)$, we must have $\Ocal(e)=D_e(e)=D_e'(e)=C(e)$, with the first equality from the definition of $\beta_{\sigma,\sigma^*}$. Therefore $X$ is oriented as $C$ in $\Ocal'_-$ and thus a conformal circuit. The second statement is the dual of the first one.
\end{proof}

Now we show that the image of $\beta_{\sigma,\sigma^*}$ is contained in the set of $(\sigma,\sigma^*)$-compatible orientations.

\begin{proposition} \label{prop:EL_pos_CCM}
Let $\Ocal$ be an orientation of $M$. If there exists a basis $B\in\Bcal(M)$ such that $B\cup f$ is a circuit conformal with $\Ocal'_-$, and $(E\setminus B)\cup g$ is a cocircuit conformal with $\widetilde{\Ocal}_-$, then $\Ocal\in\rchi(M;\sigma,\sigma^*)$.
\end{proposition}

\begin{proof}
By Lemma~\ref{lem:general_CCMO}, it suffices to show that $\Ocal'_-$ is totally cyclic and $\widetilde{\Ocal}_-$ is acyclic. Suppose $D$ is a signed cocircuit conformal with $\Ocal'_-$. Since $B$ is also a basis of $M'$ (see the proof of Lemma~\ref{lem:generic_ex_circuit}), $X:=\underline{D}\cap B$ is non-empty, but then $X$ will be simultaneously in the totally cyclic part and acyclic part of $\Ocal'_-$, contradicting \cite[Corollary 3.4.6]{BLSWZ_book}. The dual statement can be proven similarly.
\end{proof}

To prove Theorem~\ref{thm:main} via  Theorem~\ref{thm:main_actual} we need one additional construction.

\begin{lemma} \cite[Lemma 1.10]{Santos_book} \label{lem:common_ex_lifting}
Let $M'$ and $\widetilde{M}$ be oriented matroids with respective ground sets $E\sqcup\{f\}$ and $E\sqcup\{g\}$ and such that $M:=M'\setminus f = \widetilde M/g$.
Then there exists an oriented matroid $\widetilde{M}'$ on ground set $E\sqcup\{f,g\}$ such that $M'=\widetilde{M}'/g$ and $\widetilde{M}=\widetilde{M}'\setminus f$.

If $M'$ and $\widetilde{M}$ are generic as an extension and lift of $M$, then $\widetilde M'$ is a generic extension (resp. lift) of $\widetilde M$ (resp. of $M'$).
\end{lemma}

\begin{proof}[Sketch of proof]
%\paco{added sketch, to show that $f$ and $g$ can be inseparable as needed in section \ref{sec:activity}}
%Let M be an oriented matroid on a set E and let a ∈ E. Let (M/a) ∪ p be an extension of the contraction M/a. 
Every cocircuit $C$ of $\widetilde M$ vanishing on $g$ is a cocircuit of $\widetilde M/g =M'\setminus f$ and, hence, it extends to a unique cocircuit (that we still denote $C$) of $M'$. Thus, the following is a well-defined  cocircuit signature on $\widetilde M$: $\sigma^*(C) = C(g)$ if $C(g) \ne 0$ and $\sigma^*(C) = C(f)$ if $C(g) = 0$. 

That this cocircuit signature defines an extension $\widetilde M'$ of $\widetilde M$ is proven in \cite[Lemma 1.10]{Santos_book}. Calling $f$ the new element of this extension it is obvious that $\widetilde{M}=\widetilde{M}'\setminus f$ and it is easy to see that $M'=\widetilde{M}'/g$.

If $M'$ is a generic extension of $M$ then $\sigma^*$ is generic:  the only way to get a zero in $\sigma^*$ is when $\sigma^*(C) =C(f) =0$ for a cocircuit of $M'\setminus f$, but this violates genericity of $M'$.
Similarly, for $\widetilde M'$ not to be a generic lift of $M'$ there should be a cocircuit in $\widetilde M'$ using $g$ and with dependent complement. The complement cannot contain $f$, because genericity of $f$ implies that all dependent sets containing $f$ are spanning. Hence, the cocircuit (deleting $f$ from it) is also a cocircuit in $\widetilde M'\setminus f =\widetilde M$. But this violates genericity of $\widetilde M$
\end{proof}

\begin{remark}
    \label{rem:inseparable}
%\paco{added  remark}
The oriented matroid $\widetilde M'$ in this lemma is not unique, as can be easily understood looking at Figure~\ref{fig:PSA}: one can use instead of $f$ any other pseudo-line intersecting $g$ in the same points as $f$. (The linear programming  
interpretation is that $f$ does not only play the role of an objective function $c$ but of an affine hyperplane $\{c\cdot x = b\}$. Different choices of $b$ produce different oriented matroids $\widetilde M'$, but the same linear program).

Our proof of Theorem~\ref{thm:main_actual} works regardless of the choice of $\widetilde M'$. However, in Section~\ref{sec:activity} we need the following additional property of the $\widetilde M'$ constructed in the proof of the lemma: $f$ and $g$ are (positively) inseparable in it, meaning that there is no signed cocircuit having them in opposite sides.
\end{remark}

\begin{proof}[Proof of Theorem~\ref{thm:main_actual}]
``Uniqueness''. Suppose both $B_1$ and $B_2$ are bases satisfying the condition. Let $C_1,C_2$ be the signed circuits of $M'$ obtained from restricting $\Ocal'_-$ to $B_1\cup f$ and $B_2\cup f$, respectively; let $D_1,D_2$ be the signed cocircuits of $\widetilde{M}$ obtained from restricting $\widetilde{\Ocal}_-$ to $(E\setminus B_1)\cup g$ and $(E\setminus B_2)\cup g$, respectively. Let $\widetilde{M}'$ be the oriented matroid containing both $M'$ and $\widetilde{M}$ as guaranteed by Lemma~\ref{lem:common_ex_lifting} and consider the lifting $\widetilde{C_1}$ of $C_1$ in $\widetilde{M}'$.

Case I: $\widetilde{C_1}(g)=+$. Let $D'_1,D'_2$ be the extensions of $D_1,D_2$ in $\widetilde{M}'$. We must have $D'_1(f)=D'_2(f)=-$ by orthogonality, which in turn forces the lifting  $\widetilde{C_2}$ of $C_2$ to take value $+$ at $g$. Apply the circuit elimination axiom to $\widetilde{C_1}$ and $-\widetilde{C_2}$ and eliminate $f$. Denote by $C$ the resulting signed circuit. We have $\underline{C}\cap\underline{D'_1}\subset (B_2\setminus B_1)\cup g$, but $C$ is conformal with $-D'_1$ over $B_2\setminus B_1$ as $D'_1|_{B_2\setminus B_1}=\Ocal|_{B_2\setminus B_1}=C_2|_{B_2\setminus B_1}$, so $C(g)=D'_1(g)=-$ by orthogonality. However, the same orthogonality argument applied to $C$ and $D'_2$ implies that $C(g)=-D'_2(g)=+$, a contradiction.

Case II: $\widetilde{C_1}(g)=-$. The analysis is similar to Case I.

Case III: $\widetilde{C_1}(g)=0$. This case is also impossible as $\widetilde{C_1}$ cannot be orthogonal to $D'_1$ and $D'_2$.\\

``Existence''. Let $\Ocal\in\rchi(M;\sigma,\sigma^*)$. By reorienting $M$ if necessary, we may assume $\Ocal\equiv +$. For the sake of matching convention in the literature, we also reorient $f,g$ in $\widetilde{M}'$, so the all positive orientation $\Ocal'_+$ of $M'$ is totally cyclic and the all positive orientation $\widetilde{\Ocal}_+$ is acyclic by Lemma~\ref{lem:general_CCMO}. Now we consider the oriented matroid program $\Pcal:=(\widetilde{M}', g, f)$.

This oriented matroid program $\Pcal$ is both feasible and bounded from our assumption on $\widetilde{\Ocal}_+$ and $\Ocal'_+$: $\widetilde{\Ocal}_+$ itself is a (full-dimensional) feasible region of $\widetilde{M}$;
%, which corresponds to a (full-dimensional) feasible region
any positive circuit of $M'$ whose support is of the form $B\cup f$, with $ B\in\Bcal(M)$, is a bounded cone. %provides a bounded cone $B$ containing the feasible region.
By Theorem~\ref{thm:OMP}, $\Pcal$ has an optimal solution $Y$, which is a covector of $\widetilde{M}'$.

By definition, $Y$ is feasible and optimal. Since $Y$ is a covector containing $g$ in $\widetilde{M}'$, we have that $Y\setminus f$ is a covector of $\widetilde{M}$ containing $g$. So $\underline{Y}\setminus f$ contains a cocircuit (of $\widetilde{M}$), whose support is of the form $(E\setminus B_0)\cup g$ for some $B_0\in\Bcal(M)$ by Lemma~\ref{lem:generic_ex_circuit}. If the containment is proper, then $\underline{Y}\setminus f$ contains some cocircuit $\underline{Z_0}$ of $M$. Since the extension $M'$ is generic, the (ordinary matroidal) extension $\underline{Z_0'}$ of $\underline{Z_0}$ in $M'$ contains $f$. Write $Z_0'$ as the signed cocircuit of $M'$ (hence $\widetilde{M}'$) supported on $\underline{Z_0'}$ in which $Z_0'(f)=+$. Now we have a contradiction as $Y\circ Z_0'|_E$ is non-negative. Therefore $\underline{Y}\setminus f=(E\setminus B_0)\cup g$, and it is a cocircuit of $\widetilde{M}$. We claim that $B_0$ is the basis of $M$ we want.

The second assertion is immediate as $Y|_{E\cup g}$ is non-negative. By Lemma~\ref{lem:generic_ex_circuit}, $B_0\cup f$ is a circuit of $M'$. Denote by $C'$ the signed circuit of $M'$ supported on $B_0\cup f$ such that $C'(f)=+$, it remains to show $C'$ is non-negative. Suppose $C'(e)=-$. Let $\underline{Z_e}$ be the fundamental cocircuit of $e$ with respect to $B_0$ in $M$, and let $\underline{Z_e'}$ be its (ordinary matroidal) extension in $\underline{M'}$. Since the extension is generic, $f\in\underline{Z_e'}$.
Let $Z_e'$ be the signed cocircuit of $M'$ (hence $\widetilde{M}'$) supported on $\underline{Z_e'}$ in which $Z_e'(f)=+$. From the choice of $Z_e'$, $\underline{Z_e'}\cap\underline{C'}=\{e,f\}$, so $Z_e'(e)=+$ by orthogonality. In particular, $Y\circ Z_e|_E$ is non-negative, which is a contradiction. Therefore $B_0\cup f$ is a positive circuit of $\Ocal'_+$ as well.
\end{proof}

\begin{proof}[Proof of Theorem~\ref{thm:main}]
By Proposition~\ref{prop:beta_EL_pos} and \ref{prop:EL_pos_CCM}, every orientation in the image of $\beta_{\sigma,\sigma^*}$ is $(\sigma,\sigma^*)$-compatible. Injectivity follows from Proposition~\ref{prop:beta_EL_pos} and the uniqueness part of Theorem~\ref{thm:main_actual}. Surjectivity follows from Proposition~\ref{prop:beta_EL_pos} and the existence part of Theorem~\ref{thm:main_actual}.
\end{proof}

\begin{remark}
While the map described in Theorem~\ref{thm:main} is very simple and combinatorial, describing an efficient combinatorial inverse appears to be a difficult task even for special cases such as graphical matroids.
In fact, our proof of Theorem~\ref{thm:main} shows that computing the inverse map is essentially equivalent to solving generic oriented matroid programs.
%\chiho{To decide whether to keep this comment, esp. compare with the inverse algo. of G--LV}
\end{remark}

We now prove Theorem~\ref{thm:characterization}, saying
that signatures coming from extensions/liftings are the only ones for which Theorem~\ref{thm:main} works.

\begin{proof}[Proof of Theorem~\ref{thm:characterization}]
(1) implying (2) is Theorem~\ref{thm:main}, 
(2) implying (3) is the fact that every lexicographic cocircuit signature is induced by the corresponding lexicographic extension.

%\paco{edited the rest of this proof, for clarity of exposition}
To prove (3) implies (1), we assume that (1) fails and 
use Lemma~\ref{lemma:3elements}(ii). The lemma says that there is a rank 1 minor on three elements where $\sigma$ is not a lifting signature. Such minor is of the form $M/B|_A$ where $A=\{a_1,a_2,a_3\}$ and 
 $B=\{b_1,\dots,b_{r-1}\}$ are disjoint subsets of elements in $M$ such that $B_i:=B\cup a_i$ is a basis for $i=1,2,3$. In particular, $B$ is independent and, for each $i\in\{1,2,3\}$, there is a unique circuit $\underline{C_i}$ contained in $(B\cup A) \setminus \{a_i\}$ and containing $\{a_1,a_2,a_3\}\setminus \{a_i\}$.

 By reorienting the $a_i$ if necessary we assume without loss of generality that the three bases $B_i$ have the same orientation. This implies that 
 %the unique cocircuit $\underline{D}$ contained in $E\setminus B$ has the same sign (say positive) in $a_1$, $a_2$ and $a_3$, 
 for each $i \in \{1,2,3\}$ the two signed circuits supported on $\underline{C_i}$ have opposite signs in the two elements $A\setminus a_i$; let us denote $C_i$ the one that is positive in $a_{i+1}$ and negative in $a_{i-1}$, where $i$ is regarded modulo three.

With these conventions, saying that  $\sigma$
does not come from a single-element lifting of $M/B|_A$ is the same as saying that $\sigma$ chooses $C_i$ as the signed version of $\underline{C_i}$ for the three values of $i$  (or the negation thereof; $\sigma$ chooses $-C_i$ for the three values of $i$, in which case the proof proceeds in a similar manner).
%chooses the same part (positive or negative) in the three circuits $C_1$, $C_2$ and $C_3$. Suppose without loss of generality that $\sigma$ gives positive sign to the three.

Let $\Ocal$ denote the reorientation provided by $\beta_{\sigma,\sigma^*}$ for the basis $B_1$.
We have that $\Ocal(a_{2})=-$ 
(since the fundamental circuit $C(B_1,a_2)$, oriented according to $\sigma$, equals $C_3$, which is positive at $a_1$ and negative at $a_2$)
and $\Ocal(a_3)=+$ (with the same argument on the (signed) fundamental circuit  $C(B_1,a_3)$ which is $C_2$).

It only remains to show that there is a lexicographic cocircuit signature that makes the circuit $\underline{C_1}$ conformal but not compatible with $\Ocal$; that is, with $\Ocal|_{\underline{C_1}}=-C_1$.
%\paco{corrected. It said ``compatible with $\sigma$''}
%$\sigma$ selects the wrong side of $C_1$: we have $C_1(a_2)=+$ and $C_1(a_3)=-$, which is the opposite to what $\Ocal$ has.
We construct such a $\sigma^*$ in what follows. 

None of what we said above changes under reorientation of elements in $B$, so we can assume without loss of generality that $a_2$ is the only positive element in the signed circuit $C_1$.
We then define $\sigma^*$ as the positive lexicographic extension obtained by the elements of $B$ (in an arbitrary order) followed by $a_1$. 
This choice implies that in $\Ocal$ all elements of $B$ get a positive orientation, as $b_i$ is the unique element in $C^*(B_1,b_i)\cap B_1$ for every $i=1,\ldots, r-1$. On the other hand, we said above that $\Ocal|_{\underline{C_1}}$ is negative at $a_2$ and positive at $a_3$, which is the opposite of $C_1$.
This gives $\Ocal|_{\underline{C_1}}=-C_1$ as claimed.
%In particular, the circuit $C_1\subset B\cup \{a_2,a_3\}$ is conformal with $\Ocal$ but has the wrong orientation: in $C_1$ we have that $a_2$ is the only positive element and in $\Ocal|_{\underline{C_1}}$ it is the only negative one.
\end{proof}

\section{Relation with Orientation Activity and Active Bijections} 
\label{sec:activity}

In this section we briefly review the work of Gioan and Las Vergnas on {\em active bijections}, and show how our bijections are related to theirs. We first recall two  notions of activities in (oriented) matroid theory.
For this, we fix an ordering of the ground set $E$ of an oriented matroid $M$. Put differently, $M$ is now an ordered oriented matroid.

\medskip

%A set of {\em lexicographic data} $(<,s)$ of $M$ consists of a total ordering $<$ of $E$ together with a choice of sign $s(e)\in\{+,-\}$ for every element $e$ of $E$. We fix an arbitrary set of such data for the rest of the discussion.

The first notion is the classical {\em Tutte activities}, where an element $e$ is {\em internally (respectively, externally) active} for a basis $B$ if $e\in B$ and it is the minimal element in its fundamental cocircuit (respectively, $e\not\in B$, fundamental circuit).
Denote by $I(B)$ and $E(B)$ the sets of internally and externally active elements of $B$; 
the number $\iota(B):=|I(B)|,\epsilon(B):=|E(B)|$ of internally/externally active elements is the {\em internal/external activity} of $B$.

%the number $\iota(B)$ of internally active elements (respectively, the number $\epsilon(B)$ of externally active elements) is the {\em internal (respectively, external) activity} of $B$.

The second notion is (re)orientation activity of Las Vergnas \cite{LV_Activity}, where an element of $E$ is {\em internally (respectively, externally) active} in an orientation $\Ocal$ if it is the minimal element in some cocircuit (respectively, circuit) conformal with $\Ocal$. Define $I(\Ocal),E(\Ocal),\iota(\Ocal),\epsilon(\Ocal)$ for an orientation $\Ocal$ analogously.
%I(\Ocal),E(\Ocal),

\medskip

%\Paco{Not sure I understand Theorem~\ref{thm:G_LV_bij}. We are describing a map 
%\[
%\beta:\{\text{bases with orientations of their active elements$\} \to \{$orientations}\},
%\]and we implicitly say that it is injective and that $\beta(B)$ has the same active elements and with the same orientation as $B$. But which orientations are in the image...?}

In the following statement we denote  $Y^X$ the set of maps from $X$ to $Y$, and regard $\{+,-\}^E$ as the set of orientations of $M$.

\begin{theorem} \cite{GLV_AB2b} \label{thm:G_LV_bij}
There is an explicit bijection (with explicit inverse)  
\[
\GLV: \left\{(B,\varphi):B\in\Bcal(M),\ \varphi \in \{+,-\}^{I(B)\cup E(B)}\right\}\quad
\longleftrightarrow\quad \{+,-\}^E
\]
satisfying  $I(\GLV(B,\varphi))=I(B)$, \ $ E(\GLV(B,\varphi))=E(B)$, \ and \ $\GLV(B,\varphi)(e)=\varphi(e)$ for all $e\in I(B)\cup E(B)$. 

%that takes a basis $B$, together with a specified orientation for each of its active elements, to an orientation with the same internal/external active elements, having the prescribed orientations.
\end{theorem}

In other words, $\GLV$ is a bijection that preserves active elements and their specified orientations from the basis side to the orientation side.
%\Chiho{To do: (1) Discuss active partition and reduction to $(1,0)$-activity, (2) unwrap definition and explain the 3 bijections in Gioan email, (3) explain the difference between two bijections, (4) revise the original material on active classes.}
%
The bijection in Theorem~\ref{thm:G_LV_bij} has two components. First, it decomposes the matroid into minors such that the restriction of the basis to each minor is a basis and has exactly one active element: this is the {\em active decomposition} of the matroid with respect to the basis \cite{GLV_AB2a}.
With such a decomposition and up to duality, it suffices to describe the bijection under the assumption that $\iota(B)=1,\epsilon(B)=0$. Such bases are known as {\em $(1,0)$-bases} in \cite{GLV_AB1}; their images under the bijections are the {\em $(1,0)$-orientations} that satisfy $\iota(\Ocal)=1,\epsilon(\Ocal)=0$. 
Restricted to the $(1,0)$ case, Theorem~\ref{thm:G_LV_bij} reads:

\begin{corollary} 
%\cite{GLV_AB2b} 
\label{coro:G_LV_bij}
There is a bijection between 
$(1,0)$-bases and (pairs of opposite) $(1,0)$-orientations.
\end{corollary}

The following observation is the key to relate the active bijection to our work.

\begin{lemma} \label{lem:10_basis}
Denote by $g, f$ the first and second element of $M$. Then,
every $(1,0)$-basis  contains $g$ and not $f$. 

If $M$ and is a generic extension of $M\setminus f$ and a generic lift of $M/g$ then the converse holds: every basis containing $g$ and not $f$ is a $(1,0)$-basis.
\end{lemma}

\begin{proof}
$g$ is always an active element, so $\epsilon(\Ocal)=0$ implies $g\in B$. Once we know this, $f\in B$ would imply $\iota(\Ocal)\ge 2$ because both $f$ and $g$ would be internally active, so we must have $f\not\in B$.

Now, assume genericity and let $B$ be a basis containing $g$ and not $f$. Genericity of $g$ implies that every circuit not containing $g$ is spanning. Hence, every fundamental circuit $C(B,e)$ contains $g$, and no external element is active.
Genericity of $f$ implies that every cocircuit not containing $f$ has independent complement. Hence, every fundamental cocircuit $C(B,e)$ contains $f$, and $g$ is the only internally active element.
\end{proof}

So, from now on let us denote by $g,f$ the first and the second element of an oriented matroid $M$ and assume that $M$ is a generic extension of $M\setminus f$ and a generic lift of $M/g$. Observe that if we call $M_0:=M/g\setminus f$ we are in the situation of Section \ref{sec:main} with the oriented matroids $M$ and $\widetilde M'$ of that section now denoted $M_0$ and $M$. The following statement shows that the sets of bases that appear in 
Theorem~\ref{thm:main} and Corollary~\ref{coro:G_LV_bij} biject to one another.

\begin{lemma} \label{lem:basis_corr}
There is a bijection between the bases of $M_0$, cocircuits of $M\setminus f$ that contain $g$, and $(1,0)$-bases of $M$, via $B_0\leftrightarrow(E_0\setminus B_0)\cup g\leftrightarrow B_0\cup g$.
\end{lemma}

\begin{proof}
The correspondence between the first two families is the statement of Lemma~\ref{lem:generic_ex_circuit}.
On the other hand, $(1,0)$-bases of $M$ are the bases containing $g$ and not $f$, which biject (by removing/inserting $g$) to bases of $M_0:=M/g\setminus f$.
%
%By the proof of Lemma~\ref{lem:generic_ex_circuit} (or the dual case thereof), $B:=B_0\cup g$ is a basis of $M\setminus f$ hence of $M$, which we claim that it is of $(1,0)$-activity.
%    
%Let $e\not\in B, e\neq f$ be an element of $M$.
%The fundamental circuit $C$ of $e$ with respect to $B$ (in $M$) is the lifting of the fundamental circuit $C_0$ of $e$ with respect to $B_0$ (in $M_0$), so $C$ contains $g$ by genericity, and $e$ is not active with respect to $B$.
%Similarly, the fundamental circuit $C$ of $f$ with respect to $B$ (in $M$) is the lifting of $B_0\cup f$ from $M/g$ to $M$, which also contains $g$ by the strong genericity assumption.
%Finally, for every $e\in B_0$, the fundamental cocircuit of $e$ with respect to $B$ contains $f$ as the fundamental circuit of $f$ with respect to $B$ contains $B_0$, so $e$ is not active either.
%Conversely, every $(1,0)$-basis $B$ of $M$ contains $g$ by Lemma~\ref{lem:10_basis}, so $B\setminus g$ is a basis of $M/g$, hence of $M_0$ as $f\not\in B$.
\end{proof}

Let us now look at the other side, that of orientations.
We say that an acyclic orientation ({equivalently, a \em tope}) of $M$ (or $M\setminus f$) is {\em bounded} by $g$ if every cocircuit conformal with it contains $g$.
%\chiho{Probably move to an earlier section, or at least make more coherent.}
The idea behind the ``$(1,0)$'' case of the active bijection is to interpret the $(1,0)$-orientations of $M$ as topes bounded by $g$, and to relate the $(1,0)$-bases of $M$ with signed cocircuits containing $g$, then to match them with each other by an oriented matroid programming procedure involving (but not limited to) $f$ as the objective function.

In general, the optimal solutions with respect to $f$ are not cocircuits, and further tie-breaking is needed.
However, when $g$ and $f$ are  generic, the procedure is closely related to our setting, as we now elaborate.
Besides genericity, we assume that $f$ and $g$ are positively inseparable in $M$; that is, whenever a tope is bounded by $g$, it is on the same side of $f$ as of $g$. As mentioned in Remark~\ref{rem:inseparable}, the oriented matroid constructed in the proof of Lemma~\ref{lem:common_ex_lifting} has all these properties.

In the following statement we denote $\sigma$ and $\sigma^*$ the generic circuit and cocircuit signatures producing $M\setminus f$ and $M/g$ as a lifting, respectively extension, of $M_0$.

\begin{lemma} \label{lem:orientation_corr}
There is a bijection between the $(1,0)$-orientations of $M$ on the negative side of $g$ and the set $\rchi(M_0;\sigma,\sigma^*)$ of $(\sigma,\sigma^*)$-compatible orientations of $M$, via $\Ocal\leftrightarrow\Ocal/g\setminus f$.
\end{lemma}

\begin{proof}
Suppose $\Ocal$ is a $(1,0)$-orientation of $M$ on the negative side of $g$, by the inseparable assumption, $\Ocal(f)=-$ as well.  Since $\Ocal$ has zero externally activity, it is acyclic, the same can be then said for $\Ocal\setminus f$, so by Lemma~\ref{lem:general_CCMO}, $\Ocal/g\setminus f$ is $\sigma$-compatible.
If a signed cocircuit $D$ is conformal with $\Ocal/g\setminus f$ but not $\sigma^*$, 
then by the definition of $\sigma^*$, $(D\ -)$ is conformal with $\Ocal/g$ thus with $\Ocal$, in which its smallest element is an internal active element other than $g$.

Conversely, starting with a $(\sigma,\sigma^*)$-compatible orientation $\Ocal_0$ of $M_0$, consider the orientation $\Ocal$ of $M$ obtained from setting $\Ocal(f)=\Ocal(g)=-$, and is equal to $\Ocal_0$ over $E_0$.
If $\Ocal$ is conformal with a cocircuit $D$ that does not contain $g$, then $\Ocal/g$ is also conformal with $D$, contradicting Lemma~\ref{lem:general_CCMO}, so its internal activity is $1$.
Again by Lemma~\ref{lem:general_CCMO}, $\Ocal\setminus f$ is acyclic and has zero external activity.
Therefore $\Ocal$ is of $(1,0)$-activity and is by construction on the negative side of $g$.
\end{proof}

Summing up, Lemmas~\ref{lem:basis_corr} and \ref{lem:orientation_corr} say that the bijection of Theorem~\ref{thm:main} (for the oriented matroid $M_0$) ``is the same'' as the bijection of Corollary~\ref{coro:G_LV_bij} (for $M$), which in turn is the bijection of Theorem~\ref{thm:G_LV_bij} restricted to the $(1,0)$ case and with the genericity assumptions on the first and second elements of $M$.

%Hence, starting with an oriented matroid $M_0$ together with a circuit (respectively, corcircuit) signature induced by generic single-element lifting $\widetilde{M_0}:=M_0\cup g$ (respecitvely, extension $M_0':=M_0\cup f$), one could alternatively prove Theorem~\ref{thm:main} by constructing $M$ as in Lemma~\ref{lem:common_ex_lifting} that satisfies the strong genericity assumption, order the ground set of $M$ such that $g$ and $f$ are the smallest and second smallest, and apply the  active bijection together with the correspondences in Lemma~\ref{lem:basis_corr} and Lemma~\ref{lem:orientation_corr}.
%Recall that the existence of such $M$ is stronger than the proof in Section~\ref{sec:main} needed.
%Conversely, if the smallest and second smallest elements $g,f$ of an ordered oriented matroid $M$ are sufficiently generic, in the sense the no tie-break is needed in the oriented matroid program, then the $(1,0)$-case of active bijections can be reduced to Theorem~\ref{thm:main}.

\medskip
In \cite[Section 3]{GLV_Uniform} Gioan and Las Vergnas give a simplified proof of  the active bijection for the case where $M$ is uniform, a case in which no tie-breaking is needed. Although our point of view is different, our proof (once the existence of the oriented matroid $\widetilde M'$ of Section 3 is established) is in fact quite similar to their uniform case proof. For example, the circuit/cocircuit arguments that we use in the ``existence'' part of the proof of Theorem~\ref{thm:main_actual} are  close to \cite[Lemma 3.2.3]{GLV_Uniform}.
%, i.e., the oriented matroid program is {\em non-degenerate} \cite[Definition~10.1.3 (4)]{BLSWZ_book}.
However, uniformity of $M$ is much more restrictive than the  assumption in our proof, genericity of the first two elements as a lifting and an extension respectively, a situation that we can call ``uniformoid''. In particular, our Theorem~\ref{thm:main} can be understood as on the one hand extending \cite{GLV_Uniform} to this ``uniformoid'' case and on the other hand showing that any oriented matroid $M_0$ can be embedded in an uniformoid $M$ so that the active bijection for $M$ gives, in $M_0$, a bijection between bases and $(\sigma, \sigma^*)$-compatible orientations.

%Despite the resemblance of argument,
%\paco{I am tempted to remove this paragraph}
%comparing our bijections and Gioan--Las Vergnas active bijections on the \emph{same} oriented matroid (hence not allowing the enlarging construction applied above), the two bijections are rather different.
%Active bijections take an ordering of the ground set as parameter, while our bijections can take more general parameter, thus the description of the bijective algorithms are different: the former orients elements in an iterative manner whereas ours can orient elements simultaneously.
%Moreover, active bijections are distinctively activity preserving, while ours are not in general, even if the signatures are induced by lexicographic data.

%\Chiho{Shortening the rest and refer the reader to G-LV works.}

\medskip

We conclude this section with a discussion of {\em activity classes} of an oriented matroid $M$.
An intrinsic definition can be given using the {\em active decomposition} of an orientation \cite{GLV_AB2b}, a counterpart of the aforementioned active decomposition with respect to a basis.
However, for the sake of brevity, we define an activity class to be the collection of 
$2^{\iota(B)+ \epsilon(B)}$
orientations 
\[
\big\{\GLV(B,\varphi):
%B \in \Bcal(M),\ 
\varphi:I(B)\cup E(B)\rightarrow\{+,-\}\big\}
\]
for a basis $B$.
In this way, the active bijection can be viewed as a bijection between bases and activity classes of $M$, and thus activity classes  play a similar role for active bijections as $(\sigma,\sigma^*)$-compatible orientations play for ours.
This illustrates another difference between the two bijections: we are not aware of an analogue of such notion for general $(\sigma,\sigma^*)$ in our setting.

%Let $e_{1}<\ldots<e_{\iota}$ (respectively, $e'_{1}<\ldots<e'_{\epsilon}$) be the elements that are internally (respectively, externally) active in $\Ocal$. For $k=1,2,\ldots,\iota$, let $F_k$ be the union of (the supports of) all signed cocircuits conformal with $\Ocal$ whose minimal elements are at least $e_k$; dually, for $k=1,2,\ldots,\epsilon$, let $F'_k$ be the union of (the supports of) all signed circuits conformal with $\Ocal$ whose minimal elements are at least $e'_k$. The partition $\mathcal{F}=(F_{\iota}, F_{\iota-1}\setminus F_{\iota},\ldots, F_1\setminus F_2; F'_{\epsilon},F'_{\epsilon-1}\setminus F'_{\epsilon},\ldots, F'_1\setminus F'_2)$ of $E$ is the {\em active partition} of $\Ocal$. The {\em activity class} of an orientation is the set of orientations obtained from reversing any union of components from $\mathcal{F}$, which is a well-defined equivalence relation on the set of all orientations of $M$ \cite{GLV_AB2a}.
%Another way to interpret active bijections is then bijections between the bases and activity classes of $M$, in the sense that different choices of orientations of the active elements of a basis produce orientations in the same class and vice versa.
%It can be proven that any two orientations in an activity class share the same active partition (hence the same internal and external activities) \cite{GLV_AB2a}, so activity classes are well-defined and they partition the set of orientations of $M$.

%We are not aware of an analogue of such classes for general $(\sigma,\sigma^*)$ in our setting. 
Nonetheless, we describe an instance where our bijections can deduce something about activity.
Given a pair $(\sigma,\sigma^*)$ induced by the lifting and extension with respect to the same lexicographic data, a $(\sigma,\sigma^*)$-compatible orientation is called a {\em circuit-cocircuit minimal orientation} in \cite{Bac_PO} and an {\em active fixed and dual-active fixed (re)orientation} in \cite{GLV_AB2b}. We have the following observation relating these compatible orientations and activity classes.

%On the other hand, a set of lexicographic data induces a circuit signature $\sigma_{(<,s)}$ (a dual construction gives a cocircuit signature): let $\underline{C}$ be a circuit of $\underline{M}$, and let $e$ be the minimal element in $C$ with respect to $<$, then we set $\sigma_{(<,s)}(\underline{C})$ to be the unique signed circuit $C$ supported on $\underline{C}$ such that $C(e)=s(e)$. The lifting (respectively, extension) of $M$ given by that circuit (respectively, cocircuit) signature is the {\em lexicographic extension (respectively, lifting)} induced by that lexicographic data. If $\sigma$ and $\sigma^*$ are circuit and cocircuit signatures induced by the same lexicographic data, then a $(\sigma,\sigma^*)$-compatible orientation is called a {\em circuit-cocircuit minimal orientation} in \cite{Bac_PO} and an {\em active fixed and dual-active fixed (re)orientation} in \cite{GLV_AB2b}. We have the following simple observation relating these compatible orientations and activity classes.

\begin{proposition}
Let $\sigma$ and $\sigma^*$ be as above. Then 
$\rchi(M;\sigma,\sigma^*)$ is a system of representatives of the activity classes of $M$.
\end{proposition}

\begin{proof}
An orientation $\Ocal$ is $(\sigma,\sigma^*)$-compatible if and only if every circuit or cocircuit conformal with $\Ocal$ is oriented according to the reference orientation of its minimal element, if and only if every active element of $\Ocal$ is oriented according to its reference orientation.
By Theorem~\ref{thm:G_LV_bij}, exactly one orientation within an active class has such a property.
%Within an activity class, every component of the active partition of any (hence all) orientation contains exactly one active element, so there is a unique choice of reversal for each component to guarantee such element is oriented as the data.
%Since active elements are by the definition the minimum elements of any conformal circuits and cocircuits, precisely one orientation within the class is $(\sigma,\sigma^*)$-compatible.
\end{proof}

\section{Relation to Triangulations of Lawrence Polytopes}
\label{sec:lawrence}

\begin{definition}
Let $M$ be an oriented matroid on ground set $E$ and rank $r$.
The {\em Lawrence polytope} or {\em Lawrence lifting} $\Lambda(M)$ of $M$ is an oriented matroid on the ground set $\Ecal=E\sqcup \overline{E}$, where $\overline{E}=\{\overline{e}:e\in E\}$, of rank $|E|+r$ and with the following set of signed circuits:
\[
\Ccal(\Lambda(M)):=
\{\,(C^+\cup\overline{C^-}\, ,\, \overline{C^+}\cup C^-) :
(C^+,C^-) \in \Ccal(M)\,\}.
\]
That is to say, each element $\overline e \in \overline{E}$ is 
``co-antiparallel'' (antiparallel in the dual) to the corresponding $e\in E$.
\end{definition}

Every sign vector $E\rightarrow\{+,-,0\}$ (such as a signed (co)circuit or an orientation) can be interpreted as a subset of $\Ecal$ which includes $e$ (respectively, $\overline{e}$) if $e$ is positive (respectively, negative) in the sign vector.
Moreover, given a generic circuit signature $\sigma$ and a basis $B$, we can construct a subset\footnote{In \cite{Ding_GB}, such a subset of $\Ecal$ is thought as a {\em fourientation} of $M$, a notion introduced by the first author and Sam Hopkins in \cite{BH_Four}.} $B^\sigma\subset\Ecal$ as follows: include both $e$ and $\overline{e}$ if $e\in B$, otherwise include $e$ (respectively, $\overline{e}$) if $e$ is positive (respectively, negative) in $\sigma(C(B,e))$.
Observe that such a $B^\sigma$ is a basis of $\Lambda(M)$ and, conversely, every basis can be obtained this way (see Proposition~\ref{prop:triang-Lawrence}).

%We can generalize the notion of compatibility by saying a signed (co)circuit is conformal with an subset of $\Ecal$ if the subset contains the signed (co)circuit, viewed as a subset of $\Ecal$ itself.
A cocircuit signature $\sigma^*$ is a circuit signature of $M^*$, and we can similarly construct $(E\setminus B)^{\sigma^*}\subset\Ecal$ for every basis $B$ of $M$.
By duality, $(E\setminus B)^{\sigma^*}$ is a basis of $\Lambda(M^*)$ (not to be confused with $\Lambda(M)^*$).

Ding introduced the following notion in \cite{Ding_GB2} for regular matroids, which can be steadily extended to all oriented matroids.
Let us mention that, as explained in the introduction, \cite{Ding_GB2} goes beyond the setting of triangulations.

\begin{definition} 
A circuit signature $\sigma$ of $M$ is {\em triangulating} if for any $B\in\Bcal(M)$, $B^\sigma$ does not contain any signed circuit (of $M$) of the form $-\sigma(\underline{C})$, interpreted as a subset of $\Lambda(M)$.
%\paco{what is $C$ in this definition?}
A triangulating cocircuit signature can be defined analogously.
\end{definition}

The orientation $\beta_{\sigma,\sigma^*}(B)$, interpreted as a subset of $\Ecal$, is then equal to $B^\sigma\cap (E\setminus B)^{\sigma^*}$. Ding proved that for a pair of triangulating circuit and cocircuit signatures $(\sigma,\sigma^*)$ of a regular matroid $M$, the map $B\mapsto B^\sigma\cap (E\setminus B)^{\sigma^*}$ is a bijection between $\Bcal(M)$ and $\rchi(M;\sigma,\sigma^*)$ \cite[Theorem~1.20]{Ding_GB2}.
We here show that this result is the special case of Theorem~\ref{thm:main} for regular matroids.  The proof in \cite{Ding_GB2} can be extended to show that $\beta_{\sigma,\sigma^*}$ still takes bases to $(\sigma,\sigma^*)$-compatible orientations and is injective for general oriented matroids, but \cite{Ding_GB2} needed to use the {\em circuit-cocircuit reversal system}, which is only available for regular matroids, as an intermediate object to deduce surjectivity. %(namely, to show the map is injective, and that $|\rchi(M;\sigma,\sigma^*)|=|\Bcal(M)|$)

\begin{proposition} \label{prop:tri_eq_GSE}
A circuit signature is triangulating if and only if it is induced by a generic single-element lifting.
Dually, a cocircuit signature is triangulating if and only if it is induced by a generic single-element extension.
\end{proposition}

\begin{proof}
Let $\sigma$ be a triangulating circuit signature, we verify (3) of Theorem~\ref{thm:characterization}. Let $\sigma^*$ be a cocircuit signature induced by some lexicographic extension. Suppose for some basis $B$ of $M$, $\beta_{\sigma,\sigma^*}(B)$ is compatible with some $-\sigma(\underline{C})$, i.e., $B_2^\sigma\cap (E\setminus B_2)^{\sigma^*}$ contains $-\sigma(\underline{C})$.
Then $B_2^\sigma$ itself contains $-\sigma(\underline{C})$.
%, and in turn $B_2^{-\sigma}$ contains $\sigma(C)$.
%Pick a basis $B_1$ such that $C$ is a fundamental circuit thereof, $B_1^\sigma$ must contain $\sigma(C)$ by construction (and since it does not contain $-\sigma(C)$, $B_1\neq B_2$).
%Summarizing, we have $B_1^\sigma\cap B_2^{-\sigma}$ contains $\sigma(C)$, a contradiction.

Now let $\sigma$ be a circuit signature induced by some generic single-element lifting.
Suppose $B^\sigma$ contains some $-\sigma(\underline{C})$ for some bases $B$.
Consider the lexicographic cocircuit signature $\sigma^*$ with respect to an ordering of $E$ whose elements in $B\cap \underline{C}$ are the smallest, and the reference orientation of those elements are the same as in $-\sigma(\underline{C})$, we have $(E\setminus B)^{\sigma^*}$ contains $-\sigma(\underline{C})$, so $\beta_{\sigma,\sigma^*}(B)=B^\sigma\cap (E\setminus B)^{\sigma^*}$ is compatible with $-\sigma(\underline{C})$ as well, a contradiction.
\end{proof}

When an oriented matroid is realizable, its Lawrence polytope can be realized by the vertex set of an actual polytope in Euclidean space \cite[Proposition~9.3.2]{BLSWZ_book}, which is often known as the Lawrence polytope as well.
The name triangulating signature comes from the following result of Ding:

\begin{theorem}[\protect{\cite[Theorem~1.28]{Ding_GB2}}]
Let $M$ be a regular matroid and $\sigma$ a circuit signature in it.
Then, $\{\conv{B^\sigma}:B\in\Bcal(M)\}$ is a triangulation of the Lawrence polytope if and only if $\sigma$ is triangulating.
\end{theorem}

The notion of triangulation can be generalized for any oriented matroid \cite{Santos_book}, and Ding's result can also be extended; this is the main content of the rest of this section.

\begin{definition}
A {\em triangulation} of an oriented matroid $M$ is a collection of bases $B_1,\ldots, B_t\in \Bcal(M) $ that satisfies:
\begin{enumerate}
    \item Pseudo-manifold property: for any $B_i$ and $e\in B_i$, $B_i\setminus e$ is either contained in another $B_j$, or its complement contains a positive cocircuit of $M$.
    \item Non-overlapping property: there do not exist $B_i,B_j$ and a signed circuit $C=(C^+,C^-)$ such that $C^+\subset B_i$ and $\underline{C}\setminus a\subset B_j$ for some $a\in C^+$.
\end{enumerate}
\end{definition}

We have the following description of triangulations of the Lawrence polytope of an oriented matroid.

\begin{proposition} \cite[Lemma~4.11 \& Proposition~4.12]{Santos_book}
\label{prop:triang-Lawrence}
Let $M$ be an oriented matroid.
Then the bases of $\Lambda(M)$ are precisely those of the form $B^A:=B\cup\overline{B}\cup A\cup \overline{E\setminus(A\cup B)}$ for some basis $B$ of $M$ and $A\subset E\setminus B$.

Furthermore, any triangulation of $\Lambda(M)$ contains exactly one basis of the form $B^A$ for each basis $B$ of $M$.\footnote{In particular, all the triangulations of $\Lambda(M)$ have the same number of facets, equal to the number of bases of $M$}
\end{proposition}

Therefore, as mentioned at the beginning of the section, the $B^\sigma$'s constructed from a circuit signature $\sigma$ are bases of $\Lambda(M)$.
Moreover, the definition of triangulating signature (see also the notion of {\em a triangulating atlas} as in Definition~1.5(2) of \cite{Ding_GB2}) is a reformulation of the non-overlapping property.
%We only prove one side of implication because the converse is analogous, and more importantly it follows from Corollary~\ref{coro:tri_eq_tri} with no circular argument anyway.

\begin{lemma} \label{lem:nonOC_eq_tri}
%Let $\sigma$ be a triangulating circuit signature.
%Then $\{B^\sigma:B\in\Bcal(M)\}$ is non-overlapping.
%Conversely, 
Suppose a collection $\{B^{A_B}:B\in\Bcal(M), A_B\subset E\setminus B\}$ is non-overlapping.
Then there exists a triangulating circuit signature $\sigma$ such that $B^\sigma=B^{A_B}$ for every $B$.
Conversely, every triangulating circuit signature $\sigma$ gives rise to a non-overlapping collection $\{B^\sigma:B\in\Bcal(M)\}$.
\end{lemma}

\begin{proof}
Suppose $B_1,B_2\in\Bcal(M)$ share the same circuit $\underline{C}$ as fundamental circuit, but $B_1^{A_{B_1}}$ contains the signed circuit $(C^+,C^-)$ supported on $\underline{C}$ while $B_2^{A_{B_2}}$ contains $(C^-,C^+)$.
Then $B_1^{A_{B_1}},B_2^{A_{B_2}}$ overlap on the signed circuit $(C^+\cup\overline{C^-},\overline{C^+}\cup C^-)$ with the element $a$ being the unique element in $(C^+\cup\overline{C^-})\setminus (B_2\cup\overline{B_2})$, a contradiction.
Therefore, there exists a unique way to choose a signed circuits $\sigma(\underline{C})$ supported on $\underline{C}$ such that whenever $B$ has $\underline{C}$ as a fundamental circuit, $B^{A_B}$ contains $\sigma(\underline{C})$, this defines a circuit signature.
If $\sigma$ is not triangulating and $B_1^\sigma$ contains $-\sigma(\underline{C})$, pick a basis $B_2$ such that $\underline{C}$ is a fundamental circuit thereof, so $B_2^\sigma$ contains $\sigma(\underline{C})$ from construction.
A similar argument as above shows that $B_1^\sigma,B_2^\sigma$ overlap on the signed circuit of $\Lambda(M)$ coming from $-\sigma(\underline{C})$.

Conversely, suppose $\sigma$ is triangulating, but $B_1^\sigma, B_2^\sigma$ overlap on a signed circuit $(C^+\cup\overline{C^-},\overline{C^+}\cup C^-)$ of $\Lambda(M)$ that comes from the signed circuit $C=(C^+,C^-)$ of $M$.
Without loss of generality, the element $a$ in the definition of non-overlapping property is from $C^+\subset E$, and we abuse notation to view $a$ as an element of $M$.
Since $\{e,\overline{e}\}\subset B_2^\sigma$ for every $e\in \underline{C}\setminus a$, we have $\underline{C}\setminus a\subset B_2$, and $\underline{C}$ is necessarily the fundamental circuit of $a$ with respect to $B_2$.
By the construction of $B_2^\sigma$, the sign of $a$ in $\sigma(\underline{C})$ must be negative as $\overline{a}\in B_2^\sigma$.
But then $B_1^\sigma$ contains $C^+\cup\overline{C^-}\ni a$, which is the signed circuit $-\sigma(\underline{C})$ of $M$, a contradiction.
\end{proof}

The following statement is a generalization of \cite[Theorem~1.28]{Ding_GB2}, and can be thought as a converse of \cite[Proposition~4.12]{Santos_book} in the sense that ``correct number of simplices plus non-overlapping'' implies ``triangulation''. In
\cite{Ding_GB2} this result is proven for regular matroids, where the crucial step is a volume computation to show that the total volume of the simplices corresponding to the $B^\sigma$'s is equal to the volume of the Lawrence polytope (realized as a genuine polytope), hence the non-overlapping property implies that the union of these simplices is the Lawrence polytope itself. Such an argument does not work in the non-realized setting, since there is no natural notion of ``volume''.

\begin{proposition} \label{prop:tri_eq_tri}
For each basis $B$ consider an $A_B \subset E\setminus B$. If the collection
$\{B^{A_B}:B\in\Bcal(M), A_B\subset E\setminus B\}$ satisfies the non-overlapping property, then
it is a triangulation of $\Lambda(M)$.
\end{proposition}

\begin{proof}
By Lemma~\ref{lem:nonOC_eq_tri}, there exists a triangulating circuit signature $\sigma$ that yields the collection.
Such a signature is induced from some generic single-element lifting by Proposition~\ref{prop:tri_eq_GSE}, so the collection forms a triangulation of $\Lambda(M)$ by \cite[Theorem~4.14]{Santos_book}.
\end{proof}

\begin{remark}
It is possible to prove Proposition~\ref{prop:tri_eq_tri} by verifying the pseudo-manifold property directly, which in turns gives a new proof of the generic case of the ``single-element liftings yields subdivisions of $\Lambda(M)$'' direction in \cite[Theorem~4.14(i)]{Santos_book}.

Conversely, since the collection of bases of a triangulation of $\Lambda(M)$ is non-overlapping, it is induced by a triangulating circuit signature, which comes from a generic single-element lifting.
This already gives a new proof of the generic case of the reverse direction of \cite[Theorem~4.14(i)]{Santos_book}.
By the Cayley Trick \cite{HRS-Cayley}, for realizable oriented matroids it also gives a proof to the generic case of the Bohne--Dress Theorem, stating that all zonotopal tilings of a zonotope $Z(M)$ come from single-element liftings of the oriented matroid $M$ and vice versa.
%\Chiho{Can we say the proof is new, and is there some trick to deduce the full theorem(s) from the generic case?}
\end{remark}

\section*{Acknowledgements}

%\chiho{Update acknowledgement, do we need to keep the old ones?}

Work of S. Backman is supported by a Zuckerman STEM Postdoctoral Scholarship , DFG--Collaborative Research Center, TRR 109 ``Discretization in Geometry and Dynamics'', a Simons Collaboration Gift \# 854037, and NSF Grant (DMS-2246967). Work of F. Santos is supported by grants  PID2019-106188GB-I00,   PID2022-137283NB-C21  of MCIN/AEI/10.13039/501100011033 and by pro\-ject CLaPPo (21.SI03.64658) of Universidad de Cantabria and Banco Santander.
%\paco{updated my grants}
%
%MTM2017-83750-P of the Spanish Ministry of Science and grant EVF-2015-230 of the Einstein Foundation Berlin, as well as the Clay Institute and the National Science Foundation (Grant No. DMS-1440140) while he was in residence at MSRI Berkeley, California during the Fall 2017 semester ``Geometric and Topological Combinatorics''. 
Work of C.H. Yuen was supported by the Trond Mohn Foundation project ``Algebraic and Topological Cycles in Complex and Tropical Geometries''; he also acknowledges the support of the Centre for Advanced Study (CAS) in Oslo, Norway, which funded and hosted the Young CAS research project ``Real Structures in Discrete, Algebraic, Symplectic, and Tropical Geometries'' during the 2021/2022 and 2022/2023 academic years. %Netherlands Organisation for Scientific Research Vici grant 639.033.514 and % during his affiliation to University of Bern and University of Oslo, respectively
The authors thank Emeric Gioan for explaining his work with Las Vergnas and their connection with our work, and Changxin Ding for explaining his work on triangulations of Lawrence polytopes and bijections.

\bibliographystyle{plain}

\bibliography{TopoBij}

\begin{thebibliography}{10}

\bibitem{Bac_RR}
Spencer Backman.
\newblock Riemann-{R}och theory for graph orientations.
\newblock {\em Adv. Math.}, 309:655--691, 2017.

\bibitem{Bac_PO}
Spencer Backman.
\newblock Partial {G}raph {O}rientations and the {T}utte {P}olynomial.
\newblock {\em Adv. in Appl. Math.}, 94:103--119, 2018.

\bibitem{BBY}
Spencer Backman, Matthew Baker, and Chi~Ho Yuen.
\newblock Geometric bijections for regular matroids, zonotopes, and {E}hrhart
  theory.
\newblock {\em Forum Math. Sigma}, 7:e45, 37, 2019.

\bibitem{BH_Four}
Spencer Backman and Sam Hopkins.
\newblock Fourientations and the {T}utte polynomial.
\newblock {\em Res. Math. Sci.}, 4:Paper No. 18, 57, 2017.

\bibitem{BLSWZ_book}
Anders Bj\"orner, Michel Las~Vergnas, Bernd Sturmfels, Neil White, and
  G\"unter~M. Ziegler.
\newblock {\em Oriented {M}atroids}, volume~46 of {\em Encyclopedia of
  Mathematics and its Applications}.
\newblock Cambridge University Press, Cambridge, second edition, 1999.

\bibitem{Bland-OMP}
Robert~G. Bland.
\newblock A combinatorial abstraction of linear programming.
\newblock {\em J. Combinatorial Theory Ser. B}, 23(1):33--57, 1977.

\bibitem{Bohne_thesis}
J.~Bohne.
\newblock Eine kombinatorische {A}nalyse zonotopaler {R}aumaufteilungen, 1992.
\newblock Ph.D. thesis, Universit\"{a}t Bielefeld.

\bibitem{Ding_GB2}
Changxin Ding.
\newblock A framework unifying some bijections for graphs and its connection to
  {L}awrence polytopes, 2023.
\newblock preprint \href{https://arxiv.org/abs/2306.07376}{arXiv:2306.07376}.

\bibitem{Ding_GB}
Changxin Ding.
\newblock Geometric bijections between spanning subgraphs and orientations of a
  graph.
\newblock {\em Journal of the London Mathematical Society}, 108(3):1082--1120,
  2023.

\bibitem{DMTY}
Changxin Ding, Alex McDonough, Lilla T\'othm\'er\'esz, and Chi~Ho Yuen.
\newblock A consistent sandpile torsor algorithm for regular matroids.
\newblock In preparation.

\bibitem{FL-OM}
Jon Folkman and Jim Lawrence.
\newblock Oriented matroids.
\newblock {\em J. Combin. Theory Ser. B}, 25(2):199--236, 1978.

\bibitem{Gioan_thesis}
Emeric Gioan.
\newblock Correspondance naturelle entre bases et r\'{e}orientations des
  matro\"{i}des orient\'{e}s., 2002.
\newblock Ph.D. thesis, Universit\'{e} Bordeaux 1.

\bibitem{gioan2015survey}
Emeric Gioan.
\newblock A survey on the active bijection in graphs, hyperplane arrangements
  and oriented matroids.
\newblock In {\em Workshop on the Tutte polynomial}, 2015.

\bibitem{gioan2022tutte}
Emeric Gioan.
\newblock On tutte polynomial expansion formulas in perspectives of matroids
  and oriented matroids.
\newblock {\em Discrete Mathematics}, 345(7):112796, 2022.

\bibitem{gioanactive}
Emeric Gioan and Michel Las~Vergnas.
\newblock The active bijection in graphs, hyperplane arrangements, and oriented
  matroids, 3. linear programming.

\bibitem{gioan2004active}
Emeric Gioan and Michel Las~Vergnas.
\newblock The active bijection between regions and simplices in supersolvable
  arrangements of hyperplanes.
\newblock {\em The Electronic Journal of Combinatorics}, pages R30--R30, 2004.

\bibitem{GLV_Uniform}
Emeric Gioan and Michel Las~Vergnas.
\newblock Bases, reorientations, and linear programming, in uniform and rank-3
  oriented matroids.
\newblock {\em Adv. in Appl. Math.}, 32(1-2):212--238, 2004.
\newblock Special issue on the Tutte polynomial.

\bibitem{gioan2005activity}
Emeric Gioan and Michel Las~Vergnas.
\newblock Activity preserving bijections between spanning trees and
  orientations in graphs.
\newblock {\em Discrete mathematics}, 298(1-3):169--188, 2005.

\bibitem{gioan2007fully}
Emeric Gioan and Michel Las~Vergnas.
\newblock Fully optimal bases and the active bijection in graphs, hyperplane
  arrangements, and oriented matroids.
\newblock {\em Electron. Notes Discret. Math.}, 29:365--371, 2007.

\bibitem{GLV_AB1}
Emeric Gioan and Michel Las~Vergnas.
\newblock The active bijection in graphs, hyperplane arrangements, and oriented
  matroids. {I}. {T}he fully optimal basis of a bounded region.
\newblock {\em European J. Combin.}, 30(8):1868--1886, 2009.

\bibitem{gioan2009linear}
Emeric Gioan and Michel Las~Vergnas.
\newblock A linear programming construction of fully optimal bases in graphs
  and hyperplane arrangements.
\newblock {\em Electronic Notes in Discrete Mathematics}, 34:307--311, 2009.

\bibitem{GLV_AB2a}
Emeric Gioan and Michel Las~Vergnas.
\newblock The {A}ctive {B}ijection 2.a - {D}ecomposition of {A}ctivities for
  {M}atroid {B}ases, and {T}utte {P}olynomial of a {M}atroid in terms of {B}eta
  {I}nvariants of {M}inors, 2018.

\bibitem{GLV_AB2b}
Emeric Gioan and Michel Las~Vergnas.
\newblock The {A}ctive {B}ijection 2.b - {D}ecomposition of {A}ctivities for
  {O}riented {M}atroids, and {G}eneral {D}efinitions of the {A}ctive
  {B}ijection, 2018.

\bibitem{gioan2019active}
Emeric Gioan and Michel Las~Vergnas.
\newblock The active bijection for graphs.
\newblock {\em Advances in Applied Mathematics}, 104:165--236, 2019.

\bibitem{gioan2018computing}
Emeric Gioan and Michel~Las Vergnas.
\newblock Computing the fully optimal spanning tree of an ordered bipolar
  directed graph.
\newblock {\em arXiv preprint arXiv:1807.06552}, 2018.

\bibitem{GZ_HPZono}
Curtis Greene and Thomas Zaslavsky.
\newblock On the {I}nterpretation of {W}hitney {N}umbers through {A}rrangements
  of {H}yperplanes, {Z}onotopes, non-{R}adon {P}artitions, and {O}rientations
  of {G}raphs.
\newblock {\em Trans. Amer. Math. Soc.}, 280(1):97--126, 1983.

\bibitem{HRS-Cayley}
Birkett Huber, J\"{o}rg Rambau, and Francisco Santos.
\newblock The {C}ayley trick, lifting subdivisions and the {B}ohne-{D}ress
  theorem on zonotopal tilings.
\newblock {\em J. Eur. Math. Soc. (JEMS)}, 2(2):179--198, 2000.

\bibitem{TL_GB}
Tam\'{a}s K\'{a}lm\'{a}n and Lilla T\'{o}thm\'{e}r\'{e}sz.
\newblock Hypergraph polynomials and the {B}ernardi process.
\newblock {\em Algebr. Comb.}, 3(5):1099--1139, 2020.

\bibitem{LV_ATCO}
Michel Las~Vergnas.
\newblock Acyclic and {T}otally {C}yclic {O}rientations of {C}ombinatorial
  {G}eometries.
\newblock {\em Discrete Math.}, 20(1):51--61, 1977/78.

\bibitem{LasVergnas}
Michel Las~Vergnas.
\newblock Extensions ponctuelles d'une g\'{e}om\'{e}trie combinatoire
  orient\'{e}e.
\newblock In {\em Probl\`emes combinatoires et th\'{e}orie des graphes
  ({C}olloq. {I}nternat. {CNRS}, {U}niv. {O}rsay, {O}rsay, 1976)}, volume 260
  of {\em Colloq. Internat. CNRS}, pages pp 265--270. CNRS, Paris, 1978.

\bibitem{las1984correspondence}
Michel Las~Vergnas.
\newblock A correspondence between spanning trees and orientations in graphs.
\newblock {\em Graph Theory and Combinatorics, Academic Press, London}, pages
  233--238, 1984.

\bibitem{LV_Activity}
Michel Las~Vergnas.
\newblock The {T}utte {P}olynomial of a {M}orphism of {M}atroids. {II}.
  {A}ctivities of {O}rientations.
\newblock In {\em Progress in {G}raph {T}heory ({W}aterloo, {O}nt., 1982)},
  pages 367--380. Academic Press, Toronto, ON, 1984.

\bibitem{AM_GB}
Alex McDonough.
\newblock A family of matrix-tree multijections.
\newblock {\em Algebr. Comb.}, 4(5):795--822, 2021.

\bibitem{RZ_BD_thm}
J\"urgen Richter-Gebert and G\"unter~M. Ziegler.
\newblock Zonotopal {T}ilings and the {B}ohne-{D}ress {T}heorem.
\newblock In {\em Jerusalem {C}ombinatorics '93}, volume 178 of {\em Contemp.
  Math.}, pages 211--232. Amer. Math. Soc., Providence, RI, 1994.

\bibitem{Santos_book}
Francisco Santos.
\newblock Triangulations of {O}riented {M}atroids.
\newblock {\em Mem. Amer. Math. Soc.}, 156(741):viii+80, 2002.

\bibitem{YCH_thesis}
Chi~Ho Yuen.
\newblock {\em Geometric bijections of graphs and regular matroids.}
\newblock PhD thesis, Georgia Institute of Technology, 2018.

\end{thebibliography}

\medskip

Department of Mathematics and Statistics, University of Vermont, Burlington, VT 05405, USA

Email address: \url{spencer.backman@uvm.edu}\\

Department of Mathematics, Statistics and Computer Science, University of Cantabria, Spain

Email address: \url{francisco.santos@unican.es}\\

Department of Mathematical Sciences, University of Copenhagen, Universitetsparken 5, 2100 Copenhagen, Denmark

Email address: \url{chy@math.ku.dk}

\end{document}